\documentclass{article}

\usepackage[
backend=biber,
style=numeric-comp,
sorting=nty
]{biblatex}
\usepackage[margin=1in]{geometry}
\usepackage[utf8]{inputenc}
\usepackage{csquotes}
\usepackage{amsfonts,amsthm,amsmath,amssymb,accents,mathtools,physics} 
\usepackage{mathrsfs,amsbsy} 
\usepackage{enumitem}
\usepackage{subfig}
\usepackage{graphicx}
\usepackage{float}
\usepackage{xcolor}
\usepackage{algorithm}
\usepackage[noend]{algpseudocode}
\usepackage{hyperref}
\usepackage{authblk} 

\usepackage{textcomp} 
\usepackage{algorithm}
\usepackage{algorithmicx}
\usepackage{relsize}

\usepackage{bbm}
\usepackage{dsfont}

\numberwithin{equation}{section}

\usepackage{yhmath}
 \usepackage{booktabs} 

\newtheorem{thm}{Theorem}[section]
\newtheorem{cor}{Corollary}[section]

\newtheorem{rem}{Remark}[section]

\theoremstyle{definition}

\newcommand{\X}{\pmb{\mathscr{X}}}

\newcommand{\G}{\pmb{\mathscr{G}}}
\newcommand{\E}{\pmb{\mathscr{E}}}

\newcommand{\A}{\pmb{\mathscr{A}}}
\newcommand{\B}{\pmb{\mathscr{B}}}
\newcommand{\U}{\pmb{\mathscr{U}}}
\newcommand{\C}{\pmb{\mathscr{C}}}

\newcommand{\bfG}{\mathbf{G}}
\newcommand{\bfE}{\mathbf{E}}

\newcommand{\bfA}{\mathbf{A}}
\newcommand{\bfB}{\mathbf{B}}
\newcommand{\bfU}{\mathbf{U}}
\newcommand{\bfV}{\mathbf{V}}
\newcommand{\bfF}{\mathbf{F}}
\newcommand{\bfD}{\mathbf{D}}
\newcommand{\bfK}{\mathbf{K}}
\newcommand{\bfC}{\mathbf{C}}

\newcommand{\bfP}{\mathbf{P}}
\newcommand{\bfQ}{\mathbf{Q}}
\newcommand{\bfR}{\mathbf{R}}

\newcommand{\gnote}[1]{\textcolor{blue}{#1}}

\hypersetup{
    colorlinks=true,
    linkcolor=blue,
    citecolor=blue,
    filecolor=blue,      
    urlcolor=blue,
}

\title{A low-rank, high-order implicit-explicit integrator for three-dimensional convection-diffusion equations}
\author[1]{Joseph Nakao}
\author[2]{Gianluca Ceruti}
\author[2]{Lukas Einkemmer}
\affil[1]{Department of Mathematics and Statistics, Swarthmore College, Swarthmore, PA, USA}
\affil[2]{Department of Mathematics, University of Innsbruck, Innsbruck, Tyrol, Austria}
\date{}

\begin{document}

\maketitle


\begin{abstract}
\noindent This paper presents a rank-adaptive implicit-explicit integrator for the tensor approximation of three-dimensional convection-diffusion equations. In particular, the recently developed Reduced Augmentation Implicit Low-rank (RAIL) integrator is extended from the two-dimensional matrix case to the three-dimensional tensor case. The solutions are approximated using a Tucker tensor decomposition. The RAIL integrator first discretizes the partial differential equation fully in space and time using traditional methods. Here, spectral methods are considered for spatial discretizations, and implicit-explicit Runge-Kutta (IMEX RK) methods are used for time discretization. At each RK stage: the bases computed at the previous stages are augmented and reduced to construct projection subspaces. After updating the bases in a dimension-by-dimension manner, a Galerkin projection is performed to update the coefficients stored in the core tensor. As such, the algorithm balances high-order accuracy from spanning as many bases as possible from previous stages, with efficiency from leveraging low-rank structures in the solution. A post-processing step follows to maintain a low-rank solution while conserving mass, momentum, and energy. We validate the proposed method on a number of convection-diffusion problems, including a Fokker-Planck model, and a 3d viscous Burgers' equation.
\end{abstract}

\noindent\textbf{Keywords:} low-rank, Basis Update and Galerkin (BUG), Tucker decomposition, implicit-explicit method, convection-diffusion equation\\
\noindent\textbf{AMS Subject Classifications:} 65M06


\section{Introduction}

High-dimensional time-dependent partial differential equations (PDEs) play a central role in modeling complex physical phenomena, including plasma physics, heat and mass transport, reactive flows, and quantum dynamics. However, developing efficient structure-preserving numerical methods to solve such problems remains a formidable challenge. In particular, most standard grid-based discretization techniques suffer from the \textit{curse of dimensionality} -- the number of degrees of freedom grows exponentially as the number of dimensions increases. Recently, several works have exploited low-rank structures in high-dimensional PDE solutions to reduce the storage complexity, hence mitigating the curse of dimensionality, and increasing computational efficiency. Many such works have relied on low-rank decompositions of the high-dimensional solutions stored in multi-index arrays/tensors. A novel low-rank method was proposed in \cite{nakao2025reduced} for solving two-dimensional advection-diffusion and Fokker-Planck equations with high-order accuracy and mass conservation; this method is referred to as the (2d) Reduced Augmentation Implicit Low-rank (RAIL) method. In this paper, we extend the RAIL technique within the framework of three-dimensional solutions stored in a low-rank Tucker tensor decomposition. Our proposed method is low-rank, mass, momentum, and energy conservative, high-order accurate, and can be combined with implicit-explicit (IMEX) time discretizations. Whereas the 2d-RAIL paper \cite{nakao2025reduced} only considered linear advection-diffusion equations, we consider here the more general class of convection-diffusion equations
\begin{equation}\label{eq: conv_diff_eqn}
u_t + \nabla\cdot \mathbf{F}(u) = d\Delta u + c(\mathbf{x},t),\qquad\mathbf{x}\in\mathbf{\Omega}\subset\mathbb{R}^3,\quad t>0,
\end{equation}
for which a low-rank ansatz is imposed on the solution. Here, $\mathbf{F}(u)$ is a convex flux function, $d>0$ is the diffusion coefficient, and $c(\mathbf{x},t)$ is a source term. We assume that the solution, flow field, and source term can all be well approximated by low-rank functions.

While approaches based on the Proper Orthogonal Decomposition (POD), where snapshots of the dynamics are computed and subsequently compressed offline to obtain a reduced-order model of the system, have been successfully applied in various contexts, an alternative approach has emerged over the past decade. In this new approach, a low-rank decomposition of the solution is performed instantaneously, with the reduced space describing the system dynamics being updated at each time step. Many of the recent low-rank methods for solving time-dependent PDEs fall into two categories: \textit{step-and-truncate (SAT)} approach, or \text{dynamical low-rank (DLR)} approach. Both methods have found particular success in plasma physics and kinetic simulations \cite{einkemmer2025review}, as well as quantum mechanics \cite{Lubich2015a,Meyer2009}, radiative transfer \cite{PengDLR2020spherical,Coughlin2022,Kusch2022power}, and biology \cite{Jahnke2008,Kusch2021radiation,Prugger2023,einkemmer2024chemical,einkemmer2024reaction}.

The SAT approach fully discretizes the dynamics in both space and time --updating the entire low-rank solution-- followed by a truncation procedure at the end of each time-step to maintain a low-rank solution \cite{dolgov2014low,ehrlacher2017dynamical,rodgers2020step, rodgers2022adaptive, rodgers2024tensor, kormann2017low, guo2024local, guo2024conservative}. The forward stepping of the SAT approach often increases the rank due to the addition of many bases, hence the need for an efficient truncation procedure. Many of these recent works were applied to high-dimensional time-dependent kinetic simulations in the context of explicit schemes.

In the DLR approach \cite{Meyer1990,Lubich2008,koch2007dynamical}, the spatially discretized solution is decomposed into time-dependent low-rank factors describing the one-dimensional spatial bases, and a core tensor (matrix in the two-dimensional case) carrying coefficients describing the interactions between the bases. Broadly speaking, a set of time-continuous differential equations for the low-rank factors is derived by projecting the update onto the tangent space of a low-rank manifold. Solving these differential equations, we obtain the updated solution. The resulting differential equations in the original approach \cite{koch2007dynamical} are ill-conditioned in many situations, and regularization is needed \cite{Kieri2016,Nonnenmacher2008}. Substantial efforts over the past decade have focused on developing robust DLR integrators to address this issue, collectively known as the Projector-Splitting Integrator (PSI) \cite{lubich2014projector} and Basis-Update and Galerkin (BUG) integrator \cite{ceruti2022unconventional, ceruti2022rank, ceruti2024parallel}. The DLR and BUG approaches have been extended to high-order tensors \cite{koch2010dynamical,ceruti2022rank,ceruti2021time,ceruti2023rank}. Similar methods include the retraction-based DLR integrators \cite{charous2023dynamically,kieri2019projection,seguin2024low}, and the projected exponential methods \cite{Carrel2023}. We also note extensions to these lines of research include discrete empirical interpolation method (DEIM) and randomized projection methods \cite{ dektor2025collocation, ghahremani2024deim, dektor2024interpolatory,donello2023oblique, carrel2024randDLR,lam2024randomized}. DLR methods have been extended to higher than first order accuracy in \cite{Cassini2022,Einkemmer2018,ceruti2024robust,seguin2024low,kusch2025second,nobile2025robust,kieri2019projection}.

Despite the progress of low-rank explicit time integrators, there remains a great need for high-order low-rank implicit integrators, for which several methods have been developed \cite{nakao2025reduced,el2024krylov,kahza2024sylvester,sands2024high,naderi2025cross}. Second-order error bounds have been derived for implicit DLR methods \cite{ceruti2024robust,kusch2025second}, and higher than second-order accuracy has been numerically observed in \cite{nakao2025reduced,el2024krylov,kahza2024sylvester,naderi2025cross}. Most of the implicit low-rank methods have been developed for the matrix setting \cite{rodgers2023implicit, naderi2025cross, sutti2024implicit, el2024krylov, kahza2024sylvester, appelo2025robust, nakao2025reduced, li2024high, meng2024preconditioning, ding2021dynamical, einkemmer2021asymptotic, sands2024high, frank2025asymptotic}, with some being extended to the tensor setting \cite{kahza2024sylvester,ceruti2022rank,ceruti2021time}. 

The RAIL approach, originally introduced in the matrix setting in \cite{nakao2025reduced}, bridges the SAT and DLR approaches, although it adopts a fundamentally different perspective than DLR methods. Rather than starting from a projected evolution equation on the tangent space of the low-rank manifold, RAIL begins by fully discretizing the full-order problem in space and time using a Runge-Kutta (RK) scheme, much like the SAT approach. Each RK stage is then interpreted as a \textit{local} approximation of the full solution. Instead of evaluating these intermediate stages in full dimension, RAIL computes and, where necessary, replaces them with their low-rank representations by applying a BUG-type procedure at each stage, where BUG is not interpreted as a time-integration method but rather as a robust retraction strategy that preserves high-order time accuracy, as recently observed in \cite{seguin2024low}, thereby maintaining low-rank structure and accuracy throughout the integration process.

This strategy sets RAIL apart from both SAT and BUG integrators. Unlike BUG, it is not inherently constrained by the order-reducing splitting errors introduced by geometric projection, and it offers greater flexibility in constructing time-dependent projection subspaces at \textit{each} RK stage. Although the subspaces used in RAIL resemble those in BUG, the methodology allows the intermediate stages to be updated more generally, allowing for greater flexibility in how the projection spaces evolve during the integration, albeit at the cost of increased intermediate rank in the higher-dimensional setting. In this way, RAIL seeks to balance the simplicity of SAT with the geometric robustness of DLR-based integrators, while supporting high-order accuracy and enhanced adaptability. High-order accuracy is numerically observed in our experiments. The extension to Tucker tensors in the present paper can efficiently solve three-dimensional problems with low-rank structure while conserving macroscopic quantities, but the underlying ideas naturally extend to more intricate tree tensor networks such as those in \cite{ceruti2021time,ceruti2023rank}. In particular, the recursive structure of higher-dimensional decompositions -such as the Tensor Train and Hierarchical Tucker formats- builds upon operations defined in the Tucker framework. As such, the present work represents not only a practical solution for three-dimensional problems, but also an essential intermediate step towards scalable and high-order low-rank integration for higher-dimensional systems that demand implicit time integration.

The present manuscript is organized as follows. In Section 2, we briefly review the Tucker decomposition for third-order tensors. In Section 3, we introduce the 3d-RAIL scheme for convection-diffusion equations. The first-order scheme, high-order extension, and stability and consistency are presented. In Section 4, we present several numerical experiments. Conclusions are made in Section 5, and the Appendix follows afterwards.


\section{The Tucker decomposition of a third-order tensor}\label{sec: tuckerdecomp}

In this section, we provide a brief overview of the low-rank tensor decomposition used to store the solution. Throughout this paper, we follow the notation used in \cite{kolda2009tensor}. Vectors (first-order tensors/one-dimensional arrays) are denoted by boldface lowercase letters, e.g., $\mathbf{a}$. Matrices (second-order tensors/two-dimensional arrays) are denoted by boldface uppercase letters, e.g., $\mathbf{A}$. Third-order tensors/three-dimensional arrays are denoted by boldface Euler script letters, e.g., $\A$. We only concern ourselves with third-order tensors since the present paper applies tensors to efficiently solve three-dimensional equations. When considering higher-dimensional equations, e.g., during kinetic simulations, one could consider a hierarchical tree tensor network, in which case an efficient three-dimensional solver becomes a key ingredient. We refer the reader to the review papers \cite{kolda2009tensor} and \cite{grasedyck2013literature} for a systematic review of various tensor decompositions.

Three-way/third-order tensors naturally emerge from the discretization of the 3D solution to scalar partial differential equations by considering a function $u(x,y,z)$ and uniform computational grids in each dimension,
\begin{equation}\label{eq: grid}
	x_1<x_2<...<x_{N_x},\qquad y_1<y_2<...<y_{N_y},\qquad z_1<z_2<...<z_{N_z}.
\end{equation}

The function $u(x,y,z)$ \gnote{is} discretized over the tensor product of these uniform computational grids and stored in a three-dimensional array, or rather, third-order tensor, $\U\in\mathbb{R}^{N_x\times N_y\times N_z}$. The elements of $\U$ are denoted by $u_{ijk}\approx u(x_i,y_j,z_k)$, for $i=1,...,N_x$, $j=1,...,N_y$ and $k=1,...,N_z$. Due to its cubic storage complexity $\mathcal{O}(N^3)$, directly working with the full tensor $\U$ becomes impractical for large $N$, where we assume $N = N_x = N_y = N_z$. Such a situation naturally arises, for example, when using an extremely fine grid which necessitates a large $N$. Thus, just as a matrix can be decomposed and approximated using a low-rank representation obtained from a truncated singular value decomposition (SVD), a tensor can be treated similarly. Although various tensor decompositions exist, we here focus on the high-dimensional analogue of the SVD: the Tucker decomposition. This decomposition is also known as the higher-order SVD (HOSVD) or higher-order principal component analysis. The Tucker decomposition of a third-order tensor was originally proposed in \cite{tucker1963implications,tucker1966some} and extended to higher-order tensors in \cite{kapteyn1986approach}.

The Tucker decomposition decomposes a third-order tensor $\U$ into a (smaller) \textit{core tensor} which is multiplied by a matrix along each dimension. This is seen in Figure \ref{fig: Tucker} in which the core tensor is $\G$, and the matrices multiplying the core tensor (also called the \textit{factor matrices}) are $\bfV_x, \bfV_y, \bfV_z$. Although the factor matrices need not be orthonormal, this is usually desired. The factor matrices can be interpreted as one-dimensional bases with respect to each variable. Whereas, the core tensor can be interpreted as representing the amount of interaction between the one-dimensional basis vectors. However, it's important to note that the entries of the core tensor should not be thought of analogously to singular values. Unlike the singular values of a matrix, the entries of the core tensor could be negative and are not necessarily ordered, and the core tensor is generally dense.

The core tensor is size $r_x\times r_y\times r_z$, and the factor matrices are sizes $N_x\times r_x$, $N_y\times r_y$, and $N_z\times r_z$, respectively. Since the number of column vectors in each factor matrix is typically different, we say that the \textit{multilinear rank} of the tensor $\U$ is the 3-tuple $(r_x,r_y,r_z)$. Ideally, $r_x\ll N_x$, $r_y\ll N_y$ and $r_z\ll N_z$ so that the storage complexity is significantly reduced. In particular, the storage complexity of a Tucker decomposition is $\mathcal{O}(r^3+dNr)$, where $d$ is the number of spatial dimensions. This is a significant reduction from $\mathcal{O}(N^3)$ if the tensor admits a small multilinear rank. In application, the main objective is to obtain a low-multilinear rank Tucker decomposition that closely approximates the original third-order tensor. A simple example of this at the continuous level is a truncated Fourier series for a $d-$dimensional function for which the Fourier coefficients decay rapidly. Under the low-rank assumption, the Tucker decomposition for third-order tensors offers a very useful tool for reducing the storage and computational complexities.

\begin{figure}[t!]
	\centering
	\includegraphics[width=0.65\textwidth]{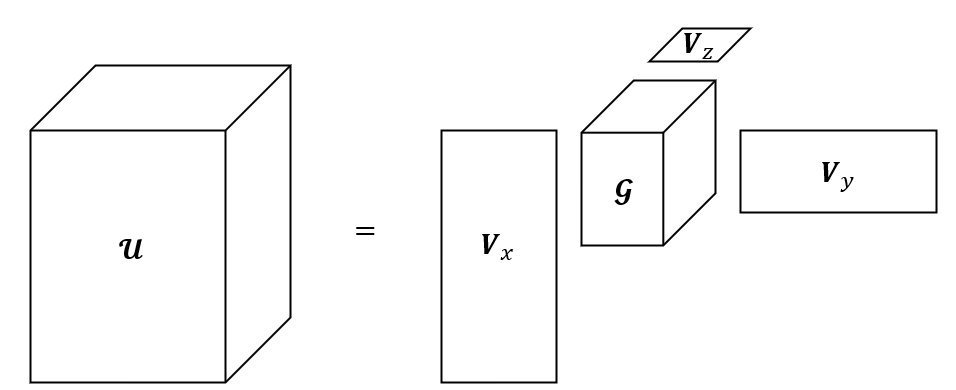}
	\caption{Tucker decomposition of a third-order tensor.}
	\label{fig: Tucker}
\end{figure}

The Tucker decomposition of a third-order tensor $\U$ is represented by
\begin{equation}\label{eq: Tucker}
	\U = \G\times_1\bfV_x\times_2\bfV_y\times_3\bfV_z,
\end{equation}
where $\times_n$ denotes the \textit{mode-$n$} product between a tensor and matrix in the $n$-th dimension. Broadly speaking, the tensor $\G$ is transformed with respect to matrix $\bfV_x$ in the $n$-th dimension. Without loss of generality, the mode-1 product between $\G$ and $\bfV_x$ is a tensor of size $N_x\times r_y\times r_z$ whose elements are
\begin{equation}
	(\G\times_1\bfV^x)_{i'jk} = \sum\limits_{i=1}^{r_x}{g_{ijk}v^x_{i'i}},\qquad i'=1,...,N_x,\ j=1,...,r_y,\ k=1,...,r_z.
\end{equation}

We note that the order of mode-$n$ products is irrelevant in a series of multiplications if the modes are distinct. Naturally, we want to analyze and separate the components of the Tucker decomposition. This can be done by flattening the tensor into a matrix, known as the \textit{mode-$n$ matricization/unfolding/flattening}. Without loss of generality, given a third-order tensor $\U\in\mathbb{R}^{N_x\times N_y\times N_z}$, one can arrange its \textit{mode-$1$ fibers} (fix all but the first index) as the columns of a matrix $\bfU_{(1)}\in\mathbb{R}^{N_x\times N_yN_z}$. The column space of $\bfU_{(1)}$ describes the dependence in $x$, and the row space of $\bfU_{(1)}$ describes the dependence in $y$ and $z$. The mode-$n$ matricizations for the Tucker decomposition \eqref{eq: Tucker} are respectively
\begin{subequations}
\begin{equation}
	\bfU_{(1)} = \bfV_x\bfG_{(1)}\big(\bfV_z\otimes\bfV_y\big)^T =  \bfV_x\bfG_{(1)}\big(\bfV_z^T\otimes\bfV_y^T\big),
\end{equation}
\begin{equation}
	\bfU_{(2)} = \bfV_y\bfG_{(2)}\big(\bfV_z\otimes\bfV_x\big)^T =  \bfV_y\bfG_{(2)}\big(\bfV_z^T\otimes\bfV_x^T\big),
\end{equation}
\begin{equation}
	\bfU_{(3)} = \bfV_z\bfG_{(3)}\big(\bfV_y\otimes\bfV_x\big)^T =  \bfV_z\bfG_{(3)}\big(\bfV_y^T\otimes\bfV_x^T\big),
\end{equation}
\end{subequations}
where $\bfG_{(n)}$ is the mode-$n$ matricization of $\G$. Here $\otimes$ denotes the Kronecker product for matrices. Similarly, one can flatten the tensor $\U$ by arranging all its mode-$n$ fibers (regardless of $n$) into a single column vector, $\text{vec}(\U)\in\mathbb{R}^{N_xN_yN_z}$, known as the \textit{vectorization} of a tensor.
\begin{equation}
	\text{vec}(\U) = \big(\bfV_z\otimes\bfV_y\otimes\bfV_x\big)\text{vec}(\G).
\end{equation} 

A natural question arises: how can we obtain a truncated Tucker decomposition of a third-order tensor? In the matrix case, one can simply truncate the singular values based on a specified tolerance to obtain the optimal low-rank approximation. However, for third-order tensors, various algorithms exist for computing a low-multilinear rank Tucker decomposition that approximates the original tensor; see \cite{kolda2009tensor} for a list of references. In the present paper, we use the \textit{truncated high-order SVD (HOSVD)} algorithm from \cite{de2000multilinear} which produces a Tucker decomposition of a specified multilinear rank $(\tilde{r}_1,\tilde{r}_2,\tilde{r}_3)$. Simply put, the truncated HOSVD takes the first $\tilde{r}_n$ left singular vectors of the matricization $\bfU_{(n)}$; note that the factor matrices are orthonormal. Then, the mode-$n$ products between $\U$ and the transposed factor matrices can be used to construct the core tensor. Alternatively, a specified tolerance on the residual can be used to determine the truncated HOSVD. The truncated HOSVD does not typically produce an optimal approximation, but it does satisfy the error bound given in Theorem \ref{thm: HOSVDbound}.

\begin{thm}[\cite{de2000multilinear,kormann2017low}]\label{thm: HOSVDbound}
Let $\U\in\mathbb{R}^{N_1\times N_2\times N_3}$ be a third order tensor, and let $\U^\star\in\mathbb{R}^{N_1\times N_2\times N_3}$ be the truncated Tucker tensor of multilinear rank $(\tilde{r}_1,\tilde{r}_2,\tilde{r}_3)$ resulting from the truncated HOSVD. Let $\mathcal{T}(\tilde{r}_1,\tilde{r}_2,\tilde{r}_3)$ denote the set of Tucker tensors of multilinear rank $(\tilde{r}_1,\tilde{r}_2,\tilde{r}_3)$. Then, the following error bound holds:
\begin{equation}
	\norm{\U-{\U^\star}}\leq\sqrt{3}\min\limits_{\X\in\mathcal{T}(\tilde{r}_1,\tilde{r}_2,\tilde{r}_3)}{\norm{\U-\X}}.
\end{equation}
\end{thm}

It's well understood that in the matrix case, the truncated SVD only conserves mass up to the truncation tolerance on the singular values. The same observation is usually seen in the tensor case. To address this issue, low-rank methods for kinetic simulations have employed strategies to truncate the numerical solution while conserving its macroscopic moments \cite{guo2024local,guo2024conservative}. We extend the Local Macroscopic Conservative (LoMaC) truncation strategy from \cite{guo2024local} to the Tucker tensor format. This procedure can truncate the solution stored in a Tucker decomposition while conserving mass, momentum, and/or energy. Depending on the model of interest, one can conserve any of these moments, if any at all. We discuss further details in the Appendix. To maintain brevity, we proceed in this paper using the non-conservative HOSVD truncation knowing that it can easily be swapped with a conservative truncation procedure.

\begin{rem}
We remind the reader that for order-$d$ tensors in higher dimensions, the storage complexity of the Tucker decomposition is $\mathcal{O}(r^d+dNr)$. Other low-rank tensor decompositions offer alternative ways to store the tensor solution, e.g., tensor train \cite{oseledets2011tensor}, and hierarchical Tucker/tree tensor networks \cite{hackbusch2009new, grasedyck2013literature}. In general, these other tensor decompositions are more advantageous for $d\geq 4$ since their storage complexities are often dominated by a $\mathcal{O}(r^3)$ term for low-rank tensors. Thus, in three dimensions, the storage complexity of the Tucker decomposition is comparable to that of these alternative decompositions, making it a suitable choice for the proposed framework.
\end{rem}

\begin{rem}
Many works in the low-rank differential equations literature use $\C$ to denote the core tensor, particularly when using tree tensor networks. To remain consistent with the notation used in \cite{kolda2009tensor} for the Tucker decomposition, we opt to use $\G$.
\end{rem}


\section{The 3d-RAIL scheme for Tucker tensor solutions}

We now introduce the proposed integrator for solving convection-diffusion equations. The diffusive terms will be evolved implicitly, while the convective terms will be handled explicitly, using implicit-explicit (IMEX) Runge-Kutta discretizations. Whereas most robust dynamical low-rank (DLR) and basis-update and Galerkin (BUG) integrators rely on time-continuous evolution equations for the low-rank factors, we take an alternative approach from \cite{nakao2025reduced} in which the original PDE of interest is instead \textit{fully} discretized in both space and time. At each Runge-Kutta stage, the fully discretized tensor equation for the original solution $\U$ is then updated in a BUG-type fashion in order to retrieve the low-rank factors.

\subsection{The semi-discrete formulation}

Discretizing the solution $u(x,y,z,t)$ over the tensor product of the uniform computational grids \eqref{eq: grid}, we assume that the numerical solution locally admits a time-dependent low-rank representation as a third-order tensor in the Tucker format,
\begin{equation}
	\U(t) = \G(t)\times_1\bfV_x(t)\times_2\bfV_y(t)\times_3\bfV_z(t),
\end{equation}
where the core tensor, factor matrices (i.e., one-dimensional bases), and multilinear rank $(r_x(t),r_y(t),r_z(t))$ are time-dependent. We first need to discretize in space, following which we will discretize in time. As such, we must discretize the flux and diffusive terms. The diffusive term, in a more general anisotropic form, is easily discretized with
\begin{align}
\big(d_1\partial_x^2+d_2\partial u_y^2+d_3\partial_y^2\big)u\qquad\longleftrightarrow\qquad
\begin{split}
	&\G(t)\times_1\bfF_x\bfV_x(t)\times_2\bfV_y(t)\times_3\bfV_z(t)\\
	&+ \G(t)\times_1\bfV_x(t)\times_2\bfF_y\bfV_y(t)\times_3\bfV_z(t)\\
	&\quad+ \G(t)\times_1\bfV_x(t)\times_2\bfV_y(t)\times_3\bfF_z\bfV_z(t),
\end{split}
\end{align}
where $\bfF_x$, $\bfF_y$, $\bfF_z$ respectively represent discretizations of the one-dimensional Laplacians $d_1\partial_x^2$, $d_2\partial_y^2$, $d_3\partial_z^2$. Although we assume the diffusion coefficients are constants, in general they could be time-dependent, and hence the differentiation matrices could also be time-dependent.

Next, we address the flux $\nabla\cdot\mathbf{F}(u)=f_1(u)_x + f_2(u)_y + f_3(u)_z.$ For simplicity, we assume that the flux in each direction is the scalar product of two functions linear in $u$. For instance, linear advection would have the form $a_i(\mathbf{x},t)u(\mathbf{x},t)$ for $i=1,2,3$. Whereas, Burgers' equation would have the form $u(\mathbf{x},t)u(\mathbf{x},t)/2$ in each direction, that is, $a_i=u/2$ for $i=1,2,3$. The transport (and source) terms need to be expressed as low-multilinear rank Tucker tensors in order to fit our projection based procedure and maintain computational efficiency. Under this assumption, the scalar flow field in each direction can be expressed as a low-multilinear rank Tucker tensor. That is, for $i=1,2,3$, the function $a_i(\mathbf{x},t)$ can be discretized as
\begin{equation}
	\A_i(t) = \G^a_i(t)\times_1\bfA_{i,x}(t)\times_2\bfA_{i,y}(t)\times_3\bfA_{i,z}(t).
\end{equation}

For $i=1,2,3$, we require a Tucker decomposition that stores the flux function $a_i(\mathbf{x},t)u(\mathbf{x},t)$ over the tensorized computational grid. Since $a_i(\mathbf{x},t)$ and $u(\mathbf{x},t)$ are each stored in Tucker tensors, we need the Tucker decomposition of the Hadamard (elementwise) product of two Tucker tensors, namely $\E_i(t)=\A_i(t)*\U(t)$. This is shown in \cite{lee2014fundamental}, resulting in the discretizations for the one-dimensional fluxes $f_1(u)$, $f_2(u)$, $f_3(u)$:
\begin{align}\label{eq: flowfield_tuckerformat}
\begin{split}
	\E_i(t) &=\G^e_i\times_1\bfE_{i,x}\times_2\bfE_{i,y}\times_3\bfE_{i,z}\\
	&\qquad\coloneqq (\G^a_i\otimes\G)\times_1(\bfA_{i,x}\odot^T\bfV_x)\times_2(\bfA_{i,y}\odot^T\bfV_y)\times_3(\bfA_{i,z}\odot^T\bfV_z),\qquad i=1,2,3,
\end{split}
\end{align}
where $\otimes$ denotes the Kronecker product for third-order tensors\footnote{This is a slight abuse of notation since the same symbol is used for the Kronecker product for matrices, but its meaning should be clear from context. Although the matrix Kronecker product is more commonly known, the tensor Kronecker product is less so. We refer the reader to \cite{ragnarsson2012structured} for a thorough treatment of this topic and structured tensor computations.}, and $\odot^T$ denotes the transpose Khatri-Rao product\footnote{\begin{equation*}
	\bfA\odot^T\bfB\coloneqq \big(\bfA^T\odot\bfB^T\big)^T = 
\Big[\mathbf{a}_{1,:}^T\otimes\mathbf{b}_{1,:}^T,\hdots,\mathbf{a}_{N,:}^T\otimes\mathbf{b}_{N,:}^T\Big]^T
=
\begin{bmatrix*}
		(\mathbf{a}_{1,:}^T\otimes\mathbf{b}_{1,:}^T)^T\\
		\hdots\\
		(\mathbf{a}_{N,:}^T\otimes\mathbf{b}_{N,:}^T)^T
\end{bmatrix*}
=
\begin{bmatrix*}
			\mathbf{a}_{1,:}\otimes\mathbf{b}_{1,:}\\
			\hdots\\
			\mathbf{a}_{N,:}\otimes\mathbf{b}_{N,:}
\end{bmatrix*}.
\end{equation*}}. As such, we get the spatial discretization of the flux term,
\begin{align}
\nabla\cdot\mathbf{F}(u)\qquad\longleftrightarrow\qquad
\begin{split}
	&\G^e_1(t)\times_1\bfD_x\bfE_{1,x}(t)\times_2\bfE_{1,y}(t)\times_3\bfE_{1,z}(t)\\
	&+ \G^e_2(t)\times_1\bfE_{2,x}(t)\times_2\bfD_y\bfE_{2,y}(t)\times_3\bfE_{2,z}(t)\\
	&\quad+ \G^e_3(t)\times_1\bfE_{3,x}(t)\times_2\bfE_{3,y}(t)\times_3\bfD_z\bfE_{3,z}(t)
\end{split}
\end{align}
where $\bfD_x$, $\bfD_y$, $\bfD_z$ respectively discretize the one-dimensional partial derivatives $\partial_x$, $\partial_y$, $\partial_z$. We note that the multilinear rank of $\E_i$ is $(r^2,r^2,r^2)$ assuming the same rank $r$ for the sake of analyzing the storage complexity. If the multilinear rank of either $\A_i$ or $\U$ is too large, this can quickly become expensive and might require further compression to maintain efficiency; such algorithms are summarized in \cite{kressner2017recompression}. Other works represent nonlinear terms using similar low-rank representations \cite{yang2025second}.

Along with the diffusive and convective terms, we also assume that the source term $c(\mathbf{x},t)$ can be represented by a low-multilinear rank Tucker tensor,
\begin{equation}
\C(t)=\G^c(t)\times_1\bfC_x(t)\times_2\bfC_y(t)\times_3\bfC_z(t).
\end{equation}

Discretizing equation \eqref{eq: conv_diff_eqn} in space according to the equations above yields the tensor differential equation

\begin{align}\label{eq: advdiff_semidiscrete}
\begin{split}
	\frac{d}{dt}\Big(\G(t)\times_1\bfV_x(t)\times_2\bfV_y(t)\times_3\bfV_z(t)\Big) &= \Big\{\G(t)\times_1\bfF_x\bfV_x(t)\times_2\bfV_y(t)\times_3\bfV_z(t)\\
	&\quad+ \G(t)\times_1\bfV_x(t)\times_2\bfF_y\bfV_y(t)\times_3\bfV_z(t)\\
	&\quad+ \G(t)\times_1\bfV_x(t)\times_2\bfV_y(t)\times_3\bfF_z\bfV_z(t)\Big\}\\
	&- \Big\{\G^e_1(t)\times_1\bfD_x\bfE_{1,x}(t)\times_2\bfE_{1,y}(t)\times_3\bfE_{1,z}(t)\\
	&\quad+ \G^e_2(t)\times_1\bfE_{2,x}(t)\times_2\bfD_y\bfE_{2,y}(t)\times_3\bfE_{2,z}(t)\\
	&\quad+ \G^e_3(t)\times_1\bfE_{3,x}(t)\times_2\bfE_{3,y}(t)\times_3\bfD_z\bfE_{3,z}(t)\Big\}\\
	&+ \G^c(t)\times_1\bfC_x(t)\times_2\bfC_y(t)\times_3\bfC_z(t).
\end{split}
\end{align}

\begin{rem}
In some cases, one might want to evolve both the transport and diffusive terms together. For instance, one might choose to cast an advection-diffusion equation in a conservative form, $u_t=\nabla\cdot(d\nabla u - \mathbf{a}u)$. In which case, $\mathbf{F}_x$, $\mathbf{F}_y$ and $\mathbf{F}_z$ could be discretizations for any appropriate finite difference scheme.
\end{rem}

\subsection{The first-order scheme using backward Euler-forward Euler discretization}\label{sec: firstorderscheme}

Discretizing the tensor differential equation \eqref{eq: advdiff_semidiscrete} using the implicit Euler-backward Euler method yields the fully discrete tensor equation
\begin{align}\label{eq: IMEX111fullydiscrete}
\begin{split}
	\G^{n+1}\times_1\bfV_x^{n+1}\times_2\bfV_y^{n+1}\times_3\bfV_z^{n+1} &= \G^{n}\times_1\bfV_x^{n}\times_2\bfV_y^{n}\times_3\bfV_z^{n}\\
	&+ \Delta t\Big\{\G^{n+1}\times_1\bfF_x\bfV_x^{n+1}\times_2\bfV_y^{n+1}\times_3\bfV_z^{n+1}\\
	&\quad\quad+ \G^{n+1}\times_1\bfV_x^{n+1}\times_2\bfF_y\bfV_y^{n+1}\times_3\bfV_z^{n+1}\\
	&\quad\quad+ \G^{n+1}\times_1\bfV_x^{n+1}\times_2\bfV_y^{n+1}\times_3\bfF_z\bfV_z^{n+1}\Big\}\\
	&- \Delta t\Big\{\G^{e,n}_1\times_1\bfD_x\bfE_{1,x}^{n}\times_2\bfE_{1,y}^{n}\times_3\bfE_{1,z}^{n}\\
	&\quad\quad+ \G^{e,n}_2\times_1\bfE_{2,x}^{n}\times_2\bfD_y\bfE_{2,y}^{n}\times_3\bfE_{2,z}^{n}\\
	&\quad\quad+ \G^{e,n}_3\times_1\bfE_{3,x}^{n}\times_2\bfE_{3,y}^{n}\times_3\bfD_z\bfE_{3,z}^{n}\Big\}\\
	&+ \Delta t\Big\{\G^{c,n+1}\times_1\bfC_x^{n+1}\times_2\bfC_y^{n+1}\times_3\bfC_z^{n+1}\Big\},
\end{split}
\end{align}
where $\U(t^{n+1}) \approx \U^{n+1} =  \G^{n+1}\times_1\bfV_x^{n+1}\times_2\bfV_y^{n+1}\times_3\bfV_z^{n+1}$. Here, we evolve the source term implicitly, although the source term can just as easily be evolved explicitly. Following the spirit of the DLR/BUG methods, the low-rank factors are updated in a projection-based fashion. The one-dimensional bases are first updated by freezing the bases in the other two dimensions, followed by a Galerkin projection in all dimensions to update the core tensor. However, unlike the DLR/BUG methods which evolve time-continuous differential equations for the low-rank factors, we instead project the fully discrete equation \eqref{eq: IMEX111fullydiscrete} onto low-dimensional subspaces spanned by some one-dimensional bases $\bfV_x^{\star,n+1}$, $\bfV_y^{\star,n+1}$, $\bfV_z^{\star,n+1}$. Since implicit integration is desired, the future bases are needed to perform the projection. Since the updated bases are initially unavailable, the bases denoted with a star $\star$ should provide good approximations to the exact bases at time $t^{n+1}$, namely, $\bfV_x(t^{n+1})$, $\bfV_y(t^{n+1})$ and $\bfV_z(t^{n+1})$. For the first-order scheme, we let $\bfV_x^{\star,n+1}\coloneqq\bfV_x^n$, $\bfV_y^{\star,n+1}\coloneqq\bfV_y^n$, and $\bfV_z^{\star,n+1}\coloneqq\bfV_z^n$.  The projected solutions in $(y,z)$, $(x,z)$ and $(x,y)$ are then respectively defined by
\begin{subequations}\label{eq: K_atanystage}
\begin{equation}\label{K1}
	\bfK_1^{n+1} \coloneqq \bfV_x^{n+1}\bfG_{(1)}^{n+1}\big((\bfV_z^{n+1})^T\bfV_z^{\star,n+1}\otimes(\bfV_y^{n+1})^T\bfV_y^{\star,n+1}\big),
\end{equation}
\begin{equation}\label{K2}
	\bfK_2^{n+1} \coloneqq \bfV_y^{n+1}\bfG_{(2)}^{n+1}\big((\bfV_z^{n+1})^T\bfV_z^{\star,n+1}\otimes(\bfV_x^{n+1})^T\bfV_x^{\star,n+1}\big),
\end{equation}
\begin{equation}\label{K3}
	\bfK_3^{n+1} \coloneqq \bfV_z^{n+1}\bfG_{(3)}^{n+1}\big((\bfV_y^{n+1})^T\bfV_y^{\star,n+1}\otimes(\bfV_x^{n+1})^T\bfV_x^{\star,n+1}\big),
\end{equation}
\end{subequations}
where each projected solution has been obtained by matricizing in the respective dimension. Since $\bfK_1^{n+1}$ has eliminated the dependence on $(y, z)$, it can be used to extract an updated orthonormal basis in $x$; the same applies to $\bfK_2^{n+1}$ and $\bfK_3^{n+1}$. Updating these projected solutions is commonly referred to as the K steps in the literature \cite{ceruti2022rank,ceruti2023rank}, in particular, the K$_1$, K$_2$ and K$_3$ steps. Without loss of generality, we detail only the K$_1$ step. Performing the mode-1 matricization on equation \eqref{eq: IMEX111fullydiscrete}, and then projecting the resulting equation in $(y,z)$ using $\bfV_y^{\star,n+1}=\bfV_y^n$ and $\bfV_z^{\star,n+1}=\bfV_z^n$, that is, multiplying on the right by $\bfV_z^n\otimes\bfV_y^n$, we have
\begin{align}\label{eq: K1_projected}
\begin{split}
	\bfV_x^{n+1}\bfG_{(1)}^{n+1}\big((\bfV_z^{n+1})^T\bfV_z^{n}\otimes(\bfV_y^{n+1})^T\bfV_y^{n}\big) &= \mathbf{W}_1^{n} + \Delta t\Big\{\bfF_x\bfV_x^{n+1}\bfG_{(1)}^{n+1}\big((\bfV_z^{n+1})^T\bfV_z^{n}\otimes(\bfV_y^{n+1})^T\bfV_y^{n}\big)\\
&\quad\quad\quad\quad\quad+ \bfV_x^{n+1}\bfG_{(1)}^{n+1}\big((\bfV_z^{n+1})^T\bfV_z^{n}\otimes(\bfF_y\bfV_y^{n+1})^T\bfV_y^{n}\big)\\
&\quad\quad\quad\quad\quad+ \bfV_x^{n+1}\bfG_{(1)}^{n+1}\big((\bfF_z\bfV_z^{n+1})^T\bfV_z^{n}\otimes(\bfV_y^{n+1})^T\bfV_y^{n}\big)\Big\}
\end{split}
\end{align}
where
\begin{align}
\begin{split}
	\mathbf{W}_1^{n} = \bfV_x^{n}\bfG_{(1)}^{n} &- \Delta t\Big\{(\bfD_x\bfE_{1,x}^n)\bfG_{1,(1)}^{e,n}\big((\bfE_{1,z}^n)^T\bfV_z^n\otimes(\bfE_{1,y}^n)^T\bfV_y^n\big)\\
	&\quad\quad+ \bfE_{2,x}^n\bfG_{2,(1)}^{e,n}\big((\bfE_{2,z}^n)^T\bfV_z^n\otimes(\bfD_y\bfE_{2,y}^n)^T\bfV_y^n\big)\\
	&\quad\quad+ \bfE_{3,x}^n\bfG_{3,(1)}^{e,n}\big((\bfD_z\bfE_{3,z}^n)^T\bfV_z^n\otimes(\bfE_{3,y}^n)^T\bfV_y^n\big)\Big\}\\
	&+ \Delta t\Big\{\bfC_x^{n+1}\bfG_{(1)}^{c,n+1}\big((\bfC_z^{n+1})^T\bfV_z^{n}\otimes(\bfC_y^{n+1})^T\bfV_y^{n}\big)\Big\}.
\end{split}
\end{align}

A standard approach to solve implicit equation \eqref{eq: K1_projected} is to cast it as a Sylvester equation for $\mathbf{K}_1^{n+1}$. However, to do so, we must further project the solution in the last two terms of equation \eqref{eq: K1_projected}. Although this projection introduces an additional error, it is on the same order as the projection error that comes from obtaining equation \eqref{eq: K1_projected}. Projecting the solution in the appropriate terms in equation \eqref{eq: K1_projected}, we get
\begin{equation}
	\mathbf{K}_1^{n+1} = \mathbf{W}_1^n + \Delta t\Big\{\bfF_x\mathbf{K}_1^{n+1} + \mathbf{K}_1^{n+1}\big(\mathbf{I}_{r_z^n\times r_z^n}\otimes(\bfF_y\bfV_y^{n})^T\bfV_y^{n}\big) + \mathbf{K}_1^{n+1}\big((\bfF_z\bfV_z^{n})^T\bfV_z^{n}\otimes\mathbf{I}_{r_y^n\times r_y^n}\big)\Big\},
\end{equation}
which can be expressed as the Sylvester equation
\begin{equation}\label{eq: K1}
	\Big(\mathbf{I}_{N_x\times N_x} - \Delta t\bfF_x\Big)\mathbf{K}_1^{n+1} - \mathbf{K}_1^{n+1}\Big(\Delta t(\bfF_z\bfV_z^{n})^T\bfV_z^{n}\oplus(\bfF_y\bfV_y^{n})^T\bfV_y^{n}\Big) = \mathbf{W}_1^{n},
\end{equation}
where $\oplus$ denotes the Kronecker sum\footnote{The Kronecker sum of two square matrices $\bfA\in\mathbb{R}^{P\times P}$ and $\bfB\in\mathbb{R}^{Q\times Q}$, denoted $\bfA\oplus\bfB\in\mathbb{R}^{PQ\times PQ}$, is the matrix $\bfA\otimes\mathbf{I}_{Q\times Q}+\mathbf{I}_{P\times P}\otimes\bfB$, where $\mathbf{I}_{P\times P}$ and $\mathbf{I}_{Q\times Q}$ are the identity matrices of sizes $P\times P$ and $Q\times Q$, respectively. Note that this is different from the direct sum of matrices which uses the same notation $\oplus$.}. We refer to equation \eqref{eq: K1} as the K$_1$ equation. Similarly, the K$_2$ and K$_3$ equations lead to the Sylvester equations
\begin{equation}\label{eq: K2}
	\Big(\mathbf{I}_{N_y\times N_y} - \Delta t\bfF_y\Big)\mathbf{K}_2^{n+1} - \mathbf{K}_2^{n+1}\Big(\Delta t(\bfF_z\bfV_z^{n})^T\bfV_z^{n}\oplus(\bfF_x\bfV_x^{n})^T\bfV_x^{n}\Big) = \mathbf{W}_2^{n},
\end{equation}
\begin{equation}\label{eq: K3}
	\Big(\mathbf{I}_{N_z\times N_z} - \Delta t\bfF_z\Big)\mathbf{K}_3^{n+1} - \mathbf{K}_3^{n+1}\Big(\Delta t(\bfF_y\bfV_y^{n})^T\bfV_y^{n}\oplus(\bfF_x\bfV_x^{n})^T\bfV_x^{n}\Big) = \mathbf{W}_3^{n}.
\end{equation}

Thus, the K steps reduce to solving three matrix Sylvester equations. Note that the K$_1$, K$_2$ and K$_3$ equations can also be solved in parallel, which accelerates computation without impacting the overall computational cost. Referring back to how $\mathbf{K}_1^{n+1}$ was defined in \eqref{K1}, the updated orthonormal basis $\bfV_x^{n+1}$ can be extracted by a reduced QR factorization, $\mathbf{K}_1^{n+1}=\mathbf{QR}\eqqcolon\bfV_x^{\ddagger,n+1}\mathbf{R}$. We toss the upper triangular matrix $\mathbf{R}$ since we only need an orthonormal basis. Similarly, we can extract the updated orthonormal bases $\bfV_y^{\ddagger,n+1}$ and $\bfV_z^{\ddagger,n+1}$ from $\mathbf{K}_2^{n+1}$ and $\mathbf{K}_3^{n+1}$, respectively. Here, the double dagger $\ddagger$ denotes the one-dimensional orthonormal bases obtained from the K steps that approximate the solution basis at $t^{n+1}$.

Following the reduced augmentation procedure from \cite{nakao2025reduced}, we propose augmenting $\bfV^{\ddagger,n+1}$ with the previous basis $\bfV^{n}$. Alternatively, we could augment $\bfV^{\ddagger,n+1}$ with $\mathbf{K}^{n+1}$, as was proposed in the original augmented BUG integrator \cite{ceruti2022rank}. Augmenting the updated and current bases together is advantageous because it ensures that the spanning subspace contains information over the entire time interval $[t^n,t^{n+1}]$. This is particularly important when the solution is rapidly evolving, for instance, over short times when solving the diffusion equation.

Unfortunately, in three dimensions this augmentation risks significantly increasing the rank since $\mathbf{K}^{n+1}$ is size $\sim N\times r^2$, and so the augmented basis is size $\sim N\times (r^2+r)$. Even for small ranks, this could quickly become costly. To remedy this issue, we reduce (truncate) the augmented basis according to a very small tolerance, usually $10^{-12}$. This tolerance is large enough to maintain a low rank, but small enough to not affect the consistency of the scheme. Although a larger tolerance could be used without affecting the overall accuracy, we risk removing basis vectors that carry physically relevant information important for the upcoming Galerkin projection. In our experience, there are often many redundant basis vectors, and this small tolerance of $10^{-12}$ does a very good job at reducing the basis. Computing the reduced QR factorizations of the augmented bases,
\begin{equation}\label{eq: redaug_order1}
	\Big[\bfV_x^{\ddagger,n+1},\bfV_x^{n}\Big] = \mathbf{Q}_x\mathbf{R}_x,\qquad\Big[\bfV_y^{\ddagger,n+1},\bfV_y^{n}\Big] = \mathbf{Q}_y\mathbf{R}_y,\qquad\Big[\bfV_z^{\ddagger,n+1},\bfV_z^{n}\Big] = \mathbf{Q}_z\mathbf{R}_z.
\end{equation}

Let $\hat{r}^{n+1}$ be the maximum of the number of singular values of $\mathbf{R}_x$, $\mathbf{R}_y$ and $\mathbf{R}_z$ larger than $10^{-12}$; we also enforce that $\hat{r}^{n+1}$ be no larger than the lengths of $\mathbf{R}_x$, $\mathbf{R}_y$ and $\mathbf{R}_z$. We then let the reduced augmented basis $\hat{\bfV}_x^{n+1}$ be $\mathbf{Q}_x$ multiplied on the right by the first $\hat{r}^{n+1}$ left singular vectors of $\mathbf{R}_x$, and similarly to obtain $\hat{\bfV}_y^{n+1}$ and $\hat{\bfV}_z^{n+1}$. With that, we can now perform a Galerkin projection onto the subspace generated by these (reduced) updated one-dimensional bases to update the core tensor $\hat{\G}^{n+1}$. This Galerkin projection is called the G step\footnote{In other works \cite{ceruti2021time,ceruti2023rank}, this is referred to as the C step. We choose to call this the G step to remain consistent with the Tucker tensor notation used in \cite{kolda2009tensor}.}, analogous to the S step in the DLR/BUG literature for solving two-dimensional problems. Here, we have used hats to denote the factorized solution produced by the K$_1$-K$_2$-K$_3$-G procedure. Performing a Galerkin projection onto the updated bases obtained via reduced augmentation,
\begin{equation}
	\text{vec}(\hat{\G}^{n+1}) \coloneqq \big((\hat{\bfV}_z^{n+1})^T\bfV_z^{n+1}\otimes(\hat{\bfV}_y^{n+1})^T\bfV_y^{n+1}\otimes(\hat{\bfV}_x^{n+1})^T\bfV_x^{n+1}\big)\text{vec}(\G^{n+1}).
\end{equation}

Similar to the K steps, we vectorize equation \eqref{eq: IMEX111fullydiscrete}, and then project the resulting equation by multiplying on the left by $\big(\hat{\bfV}_z^{n+1}\otimes\hat{\bfV}_y^{n+1}\otimes\hat{\bfV}_x^{n+1}\big)^T$ to get the tensor linear equation
\begin{align}\label{eq: Gequation}
\begin{split}
	&\Bigg\{\Big(\mathbf{I}_{\hat{r}^{n+1}\times \hat{r}^{n+1}}-\Delta t(\hat{\bfV}_z^{n+1})^T(\bfF_z\hat{\bfV}_z^{n+1})\Big)\otimes\mathbf{I}_{\hat{r}^{n+1}\times \hat{r}^{n+1}}\otimes\mathbf{I}_{\hat{r}^{n+1}\times \hat{r}^{n+1}}\\
	&\quad\quad+ \mathbf{I}_{\hat{r}^{n+1}\times \hat{r}^{n+1}}\otimes\Big(-\Delta t(\hat{\bfV}_y^{n+1})^T(\bfF_y\hat{\bfV}_y^{n+1})\Big)\otimes\mathbf{I}_{\hat{r}^{n+1}\times \hat{r}^{n+1}}\\
	&\quad\quad\quad+ \mathbf{I}_{\hat{r}^{n+1}\times \hat{r}^{n+1}}\otimes\mathbf{I}_{\hat{r}^{n+1}\times \hat{r}^{n+1}}\otimes\Big(-\Delta t(\hat{\bfV}_x^{n+1})^T(\bfF_x\hat{\bfV}_x^{n+1})\Big)\Bigg\}\text{vec}(\hat{\G}^{n+1}) = \text{vec}(\B^{n}),
\end{split}
\end{align}
where
\begin{align}
\begin{split}
	\B^{n} &= \G^n\times_1(\hat{\bfV}_x^{n+1})^T\bfV_x^n\times_2(\hat{\bfV}_y^{n+1})^T\bfV_y^n\times_3(\hat{\bfV}_z^{n+1})^T\bfV_z^n\\
	&- \Delta t\Big\{\G^{e,n}_1\times_1(\hat{\bfV}_x^{n+1})^T(\bfD_x\bfE_{1,x}^{n})\times_2(\hat{\bfV}_y^{n+1})^T\bfE_{1,y}^{n}\times_3(\hat{\bfV}_z^{n+1})^T\bfE_{1,z}^{n}\\
	&\quad\quad+ \G^{e,n}_2\times_1(\hat{\bfV}_x^{n+1})^T\bfE_{2,x}^{n}\times_2(\hat{\bfV}_y^{n+1})^T(\bfD_y\bfE_{2,y}^{n})\times_3(\hat{\bfV}_z^{n+1})^T\bfE_{2,z}^{n}\\
	&\quad\quad+ \G^{e,n}_3\times_1(\hat{\bfV}_x^{n+1})^T\bfE_{3,x}^{n}\times_2(\hat{\bfV}_y^{n+1})^T\bfE_{3,y}^{n}\times_3(\hat{\bfV}_z^{n+1})^T(\bfD_z\bfE_{3,z}^{n})\Big\}\\
	&+ \Delta t\Big\{\G^{c,n+1}\times_1(\hat{\bfV}_x^{n+1})^T\bfC_x^{n+1}\times_2(\hat{\bfV}_y^{n+1})^T\bfC_y^{n+1}\times_3(\hat{\bfV}_z^{n+1})^T\bfC_z^{n+1}\Big\}.
\end{split}
\end{align}

Unlike the matrix case, the equation for the core tensor does not satisfy a Sylvester equation \cite{nakao2025reduced}. Instead, we must solve a third-order tensor linear equation with Kronecker structure. Naively solving equation \eqref{eq: Gequation} would be prohibitively expensive due to the $r^3\times r^3$ coefficient matrix, hence motivating a more efficient solver. There are many efficient iterative solvers that exploit Kronecker and low-rank structures of the iteration matrices \cite{kressner2015truncated, kurschner2016efficient}. However, the algorithm presented in \cite{simoncini2020numerical} offers an alternative method that is particularly ideal for our situation since it is a direct solver, straightforward to understand, and easy to implement. Moreover, it does not require the use of the coefficient matrix in Kronecker form. As mentioned in \cite{simoncini2020numerical}, this direct solver can serve as a workhorse for solving reduced equations that show up in projection based procedures for large and sparse third-order tensor equations, such as the proposed implicit integrator. We emphasize that this algorithm is specifically designed for solving third-order tensor linear equations. For higher-order tensor linear equations, iterative methods might be a more suitable choice. Moreover, we emphasize that while the algorithm in \cite{simoncini2020numerical} is presented for rank-1 righthand sides of the form $\B=\mathbf{b}_1\otimes\mathbf{b}_2\otimes\mathbf{b}_3$, it can be easily extended to general (low-rank) third-order tensors $\B$. We thus present this extension as Corollary \ref{cor: 3Dsolver} in the Appendix, where we provide its proof for completeness. 

After using the direct solver described in Corollary \ref{cor: 3Dsolver} to solve for $\hat{\G}^{n+1}$, we truncate the Tucker tensor solution by using the multilinear SVD (MLSVD) \cite{de2000multilinear}, also known as the higher order SVD (HOSVD), to compress the core tensor. We use the \texttt{mlsvd} function in the MATLAB toolbox Tensorlab \cite{vervliet2016tensorlab,vervliet2016tensorlab_confpaper} to compress $\hat{\G}^{n+1}$ according to tolerance $\varepsilon$, usually between $10^{-4}$ and $10^{-8}$; we acknowledge KU Leuven as the provider of the software. The result is a (smaller) Tucker tensor of multilinear rank $(r_x^{n+1},r_y^{n+1},r_z^{n+1})$ that approximates the core tensor,
\begin{equation}
	\hat{\G}^{n+1} \approx \G^{n+1}\times_1\bfG_x^{n+1}\times_2\bfG_y^{n+1}\times_3\bfG_z^{n+1},
\end{equation}
where $\bfG_x^{n+1}$, $\bfG_y^{n+1}$ and $\bfG_z^{n+1}$ have orthonormal column vectors, and $\G^{n+1}$ is defined to be the final updated core tensor. The final updated one-dimensional bases/factor matrices of the Tucker tensor solution are then $\bfV_x^{n+1}=\hat{\bfV}_x^{n+1}\bfG_x^{n+1}$, $\bfV_y^{n+1}=\hat{\bfV}_y^{n+1}\bfG_y^{n+1}$ and $\bfV_z^{n+1}=\hat{\bfV}_z^{n+1}\bfG_z^{n+1}$. 

As mentioned in Section \ref{sec: tuckerdecomp}, the MLSVD/HOSVD procedure only conserves mass up to the truncation tolerance $\varepsilon$. We present the Local Macroscopic Conservative (LoMaC) truncation procedure from \cite{guo2024local} -extended to third-order Tucker tensors- in the Appendix. This conservative truncation procedure can conserve mass, momentum, and/or energy, depending on what the user desires based on the model being solved.

\begin{rem}
For practitioners familiar with BUG-type algorithms, it is important to note a key structural difference in the choice of projection spaces during the intermediate Runge-Kutta stages. While BUG (in higher-dimensional tree tensor formats \cite{ceruti2022rank,ceruti2021time,ceruti2023rank}) fixes both the factor matrices and the core tensor along inactive modes at each internal K-stage, 3d-RAIL allows the core tensor to remain unfrozen. This design choice increases the rank of the matrix unfolding involved in the projection step (from $r$ to $r^2$), but enables a dynamic evolution of the co-range over time. This flexibility aligns with the design principles behind the original 2d-RAIL scheme and may allow the algorithm to achieve a more compact representation over time compared to the fixed subspace used in BUG. As a result of this structural difference, 3d-RAIL remains consistent with BUG-based approximations but constitutes a distinct algorithm tailored for Tucker tensor formats. While this distinction plays no decisive role when using a first-order implicit method, it becomes critical when aiming to achieve higher-order accuracy.
\end{rem}

\subsection{The high-order scheme using implicit-explicit Runge-Kutta (IMEX RK) discretizations}

The advantage of the RAIL formulation that distinguishes it from the BUG methods is better seen in the high-order method. We begin by full discretizing the original PDE, that is, discretizing equation \eqref{eq: advdiff_semidiscrete} using a high-order IMEX RK scheme. Since we have completely discretized in time, the K1-K2-K3-G procedure can be applied at \textit{each} stage of the Runge-Kutta method, projecting onto richer subspaces at each subsequent stage. In this sense, the projection is local to each stage and the projection subspace is stage-dependent, tightly linked to the local time approximation.

\begin{table}[h!]
\begin{minipage}[b]{0.49\linewidth}
\centering
\caption{Implicit RK Scheme}
\label{table:IMEX_implicit}
\begin{tabular}{c|lllll}
    0&0&0&0&0&0\\
    $c_1$&0&$a_{11}$&0&$\hdots$&0\\
    $c_2$&0&$a_{21}$&$a_{22}$&$\hdots$&0\\
    $\vdots$&$\vdots$&$\vdots$&$\vdots$&$\ddots$&$\vdots$\\
    $c_s$&0&$a_{s1}$&$a_{s2}$&$\hdots$&$a_{ss}$\\
    \hline
    &0&$b_1$&$b_2$&$\hdots$&$b_s$
    \end{tabular}
\end{minipage}
\begin{minipage}[b]{0.49\linewidth}
\centering
\caption{Explicit RK Scheme}
\label{table:IMEX_explicit}
\begin{tabular}{c|lllll}
    0&0&0&0&$\hdots$&0\\
    $c_1$&$\tilde{a}_{21}$&0&0&$\hdots$&0\\
    $c_2$&$\tilde{a}_{31}$&$\tilde{a}_{32}$&0&$\hdots$&0\\
    $\vdots$&$\vdots$&$\vdots$&$\vdots$&$\ddots$&$\vdots$\\
    $c_s$&$\tilde{a}_{\sigma 1}$&$\tilde{a}_{\sigma 2}$&$\tilde{a}_{\sigma 3}$&$\hdots$&0\\
    \hline
    &$\tilde{b}_1$&$\tilde{b}_2$&$\tilde{b}_3$&$\hdots$&$\tilde{b}_{\sigma}$
    \end{tabular}
\end{minipage}
\end{table}

Following the notation from \cite{Ascher1997}, IMEX($s$,$\sigma$,$p$) couples an $s$ stage DIRK scheme with a $\sigma=s+1$ stage explicit RK scheme, with combined order $p$. IMEX($s$,$\sigma$,$p$) RK methods are expressed by two Butcher tableaus, one each for the implicit and explicit RK methods, as seen in Tables \ref{table:IMEX_implicit} and \ref{table:IMEX_explicit}. Discretizing the tensor differential equation \eqref{eq: advdiff_semidiscrete} in time, the equation for the intermediate solution $\U^{(k)}\approx\U(t^{(k)}=t^n+c_k\Delta t)$ at the $k$th stage is (for $k=1,2,...,s$)
\begin{align}\label{eq: IMEX_k}
\begin{split}
	\G^{(k)}\times_1\bfV_x^{(k)}\times_2\bfV_y^{(k)}\times_3\bfV_z^{(k)} &= \G^{n}\times_1\bfV_x^{n}\times_2\bfV_y^{n}\times_3\bfV_z^{n}\\
	&+ \Delta t\sum\limits_{\ell=1}^{k}a_{k\ell}\Big\{\G^{(\ell)}\times_1\bfF_x\bfV_x^{(\ell)}\times_2\bfV_y^{(\ell)}\times_3\bfV_z^{(\ell)}\\
	&\quad\quad\quad\quad\quad+ \G^{(\ell)}\times_1\bfV_x^{(\ell)}\times_2\bfF_y\bfV_y^{(\ell)}\times_3\bfV_z^{(\ell)}\\
	&\quad\quad\quad\quad\quad+ \G^{(\ell)}\times_1\bfV_x^{(\ell)}\times_2\bfV_y^{(\ell)}\times_3\bfF_z\bfV_z^{(\ell)}\Big\}\\
	&- \Delta t\sum\limits_{\ell=1}^{k}\tilde{a}_{k+1,\ell}\Big\{\G^{e,(\ell-1)}_1\times_1\bfD_x\bfE_{1,x}^{(\ell-1)}\times_2\bfE_{1,y}^{(\ell-1)}\times_3\bfE_{1,z}^{(\ell-1)}\\
	&\quad\quad\quad\quad\quad\quad+ \G^{e,(\ell-1)}_2\times_1\bfE_{2,x}^{(\ell-1)}\times_2\bfD_y\bfE_{2,y}^{(\ell-1)}\times_3\bfE_{2,z}^{(\ell-1)}\\
	&\quad\quad\quad\quad\quad\quad+ \G^{e,(\ell-1)}_3\times_1\bfE_{3,x}^{(\ell-1)}\times_2\bfE_{3,y}^{(\ell-1)}\times_3\bfD_z\bfE_{3,z}^{(\ell-1)}\Big\}\\
	&+ \Delta t\sum\limits_{\ell=1}^{k}a_{k\ell}\Big\{\G^{c,(\ell)}\times_1\bfC_x^{(\ell)}\times_2\bfC_y^{(\ell)}\times_3\bfC_z^{(\ell)}\Big\}.
\end{split}
\end{align}

We restrict ourselves to stiffly accurate DIRK methods for which $c_s=1$ and $a_{sk}=b_k$ for $k=1,2,...,s$. As such, $\U^{(s)}=\U^{n+1}$ and we do not need to compute the final stage of the RK method. Non-stiffly accurate methods can also be used, although the final stage must be computed. The high-order RAIL scheme performs the K$_1$-K$_2$-K$_3$-G (and truncate) procedure using enriched projection subspaces defined by $\bfV_{\star}^{x,(k)}$, $\bfV_{\star}^{y,(k)}$ and $\bfV_{\star}^{z,(k)}$ at each subsequent stage. The question that remains is: how do we define the approximate bases $\bfV_{\star}^{x,(k)}$, $\bfV_{\star}^{y,(k)}$ and $\bfV_{\star}^{z,(k)}$ at each stage?

K$_1$-K$_2$-K$_3$ Steps. The first-order scheme only used the reduced augmentation procedure preceding the G step, see equation \eqref{eq: redaug_order1}. For the high-order scheme, we will use the reduced augmentation procedure before the K$_1$-K$_2$-K$_3$ steps (to construct the approximate bases) \textit{and} before the G step. Without loss of generality of the dimension, at the $k$-th RK stage, we reduce the augmented basis
\begin{equation}\label{eq: redaug_RK}
	\Big[\bfV_x^{\dagger,(k)},\bfV_x^{(k-1)},...,\bfV_x^{(1)},\bfV_x^{(0)}\Big],
\end{equation}
where $\bfV_x^{(0)}=\bfV_x^{n}$, and $\bfV_x^{\dagger,(k)}$ denotes a first-order prediction at time $t^{(k)}$ computed using the first-order RAIL scheme. Similar to equation \eqref{eq: redaug_order1}, a first-order prediction at the future time $t^{(k)}$ is included in the augmented basis so that the subspace contains information over the entire time interval $[t^n,t^{(k)}]$. Note that $\bfV_x^{\dagger,(k)}$ is only needed for $k>1$ since the first stage $k=1$ is itself a backward Euler-forward Euler discretization. After computing the reduced QR factorizations of the augmented bases, the projection bases $\bfV_x^{\star,(k)}$, $\bfV_y^{\star,(k)}$, $\bfV_z^{\star,(k)}$ are constructed in the same manner as in the first-order scheme, where the singular values of the upper triangular factors $\mathbf{R}_x$, $\mathbf{R}_y$, $\mathbf{R}_z$ are truncated according to a small tolerance $10^{-12}$.

The projected solutions in $(y,z)$, $(x,z)$ and $(x,y)$ are respectively defined by \eqref{K1}, \eqref{K2} and \eqref{K3}; naturally, a superscript $(k)$ can be used to denote the projected solutions at time $t^{(k)}$. Following the same detailed procedure outlined in subsection \ref{sec: firstorderscheme} for the K$_1$ step, but instead working with equation \eqref{eq: IMEX_k}, one obtains the Sylvester equation
\begin{equation}\label{eq: K1_highorder}
	\Big(\mathbf{I}_{N_x\times N_x} - a_{kk}\Delta t\bfF_x\Big)\mathbf{K}_1^{(k)} - \mathbf{K}_1^{(k)}\Big(a_{kk}\Delta t(\bfF_z\bfV_z^{\star,(k)})^T\bfV_z^{\star,(k)}\oplus(\bfF_y\bfV_y^{\star,(k)})^T\bfV_y^{\star,(k)}\Big) = \mathbf{W}_1^{(k-1)},
\end{equation}
where
\begin{align}\label{eq: IMEX_RAIL3D_Wk-1}
\begin{split}
	\mathbf{W}_1^{(k-1)} &= \bfV_x^{n}\bfG_{(1)}^{n}\big((\bfV_z^{n})^T\bfV_z^{\star,(k)}\otimes(\bfV_y^{n})^T\bfV_y^{\star,(k)}\big)\\
&+ \Delta t\sum\limits_{\ell=1}^{k-1}a_{k\ell}\Big\{(\bfF_x\bfV_x^{(\ell)})\bfG_{(1)}^{(\ell)}\big((\bfV_z^{(\ell)})^T\bfV_z^{\star,(k)}\otimes(\bfV_y^{(\ell)})^T\bfV_y^{\star,(k)}\big)\\
&\quad\quad\quad\qquad+ \bfV_x^{(\ell)}\bfG_{(1)}^{(\ell)}\big((\bfV_z^{(\ell)})^T\bfV_z^{\star,(k)}\otimes(\bfF_y\bfV_y^{(\ell)})^T\bfV_y^{\star,(k)}\big)\\
&\quad\quad\quad\qquad+ \bfV_x^{(\ell)}\bfG_{(1)}^{(\ell)}\big((\bfF_z\bfV_z^{(\ell)})^T\bfV_z^{\star,(k)}\otimes(\bfV_y^{(\ell)})^T\bfV_y^{\star,(k)}\big)\Big\}\\
&- \Delta t\sum\limits_{\ell=1}^{k}\tilde{a}_{k+1,\ell}\Big\{(\bfD_x\bfE_{1,x}^{(\ell-1)})\bfG_{1,(1)}^{e,(\ell-1)}\big((\bfE_{1,z}^{(\ell-1)})^T\bfV_z^{\star,(k)}\otimes(\bfE_{1,y}^{(\ell-1)})^T\bfV_y^{\star,(k)}\big)\\
&\quad\quad\quad\quad\quad\qquad+ \bfE_{2,x}^{(\ell-1)}\bfG_{2,(1)}^{e,(\ell-1)}\big((\bfE_{2,z}^{(\ell-1)})^T\bfV_z^{\star,(k)}\otimes(\bfD_y\bfE_{2,y}^{(\ell-1)})^T\bfV_y^{\star,(k)}\big)\\
&\quad\quad\quad\quad\quad\qquad+ \bfE_{3,x}^{(\ell-1)}\bfG_{3,(1)}^{e,(\ell-1)}\big((\bfD_z\bfE_{3,z}^{(\ell-1)})^T\bfV_z^{\star,(k)}\otimes(\bfE_{3,y}^{(\ell-1)})^T\bfV_y^{\star,(k)}\big)\Big\}\\
&+ \Delta t\sum\limits_{\ell=1}^{k}a_{k\ell}\Big\{\bfC_x^{(\ell)}\bfG_{(1)}^{c,(\ell)}\big((\bfC_z^{(\ell)})^T\bfV_z^{\star,(k)}\otimes(\bfC_y^{(\ell)})^T\bfV_y^{\star,(k)}\big)\Big\}.
\end{split}
\end{align}

The K$_2$ and K$_3$ Sylvester equations can be derived similarly. With $\mathbf{K}_1^{(k)}$, $\mathbf{K}_2^{(k)}$ and $\mathbf{K}_3^{(k)}$ computed, a reduced QR factorization extracts orthonormal bases for each dimension. That is, $\mathbf{K}_1^{(k)}=\mathbf{QR}\eqqcolon\bfV_x^{\ddagger,(k)}\mathbf{R}$, and similarly for $\bfV_y^{\ddagger,(k)}$ and $\bfV_z^{\ddagger,(k)}$. As with the first-order scheme, a double dagger $\ddagger$ denotes the one-dimensional orthonormal bases obtained from the K steps that approximate the solution basis at $t^{(k)}$.

\begin{algorithm}[t!]
	\caption{The 3d-RAIL algorithm for convection-diffusion equations using IMEX RK methods}
	\label{algo: highorderIMEX}
		{\bf Input:} $\bfV_x^n$, $\bfV_y^n$, $\bfV_z^n$, $\G^n$ and $(r_x^n,r_y^n,r_z^n)$\\
	    {\bf Output:} $\bfV_x^{n+1}$, $\bfV_y^{n+1}$, $\bfV_z^{n+1}$, $\G^{n+1}$ and $(r_x^{n+1},r_y^{n+1},r_z^{n+1})$\\
\textbf{for} each RK stage $k=1,2,...,s$ \textbf{do}
	\begin{algorithmic}[1]
	   \State Compute the predictions $\bfV_x^{\dagger,(k)}$, $\bfV_y^{\dagger,(k)}$ and $\bfV_z^{\dagger,(k)}$; not needed for first stage.
	  \State Construct the projection bases $\bfV_x^{\star,(k)}$, $\bfV_y^{\star,(k)}$ and $\bfV_z^{\star,(k)}$.
	\State \textbf{for} each dimension $i\in\{x,y,z\}$ \textbf{do}
	\Statex\quad Compute $\mathbf{W}_i^{(k-1)}$, e.g., \eqref{eq: IMEX_RAIL3D_Wk-1}.
	\Statex\quad Solve the K$_i$ equation, e.g., \eqref{eq: K1_highorder} for $\mathbf{K}_i^{(k)}$.
	\Statex\quad Compute and store the update basis $\bfV_i^{\ddagger,(k)}$.
        \State Construct the bases $\hat{\bfV}_x^{(k)}$, $\hat{\bfV}_y^{(k)}$ and $\hat{\bfV}_z^{(k)}$.
        \State Compute $\B^{(k-1)}$ \eqref{eq: IMEX_RAIL3D_Gk-1}, and solve the G equation \eqref{eq: G_highorder} for $\hat{\G}^{(k)}$.
	\State Truncate the solution using the truncated HOSVD/MLSVD (or LoMaC).
	\State Store $\bfV_x^{(k)}$, $\bfV_y^{(k)}$, $\bfV_z^{(k)}$ and $\G^{(k)}$.
	\end{algorithmic}
Store the final solution $\bfV_x^{n+1}$, $\bfV_y^{n+1}$, $\bfV_z^{n+1}$, $\G^{n+1}$ and $(r_x^{n+1},r_y^{n+1},r_z^{n+1})$.
\end{algorithm}

G Step. Replacing $\bfV^{\dagger,(k)}$ with $\bfV^{\ddagger,(k)}$ in the augmented matrices, e.g., equation \eqref{eq: redaug_RK} for $x$-dimension, we perform the reduced augmentation procedure to obtain the pre-truncated bases $\hat{\bfV}^{(k)}$. We again let hats denote the factorized solution produced by the K$_1$-K$_2$-K$_3$-G procedure. Performing a Galerkin projection onto these updated bases,
\begin{equation}
	\text{vec}(\hat{\G}^{(k)}) \coloneqq \big((\hat{\bfV}_z^{(k)})^T\bfV_z^{(k)}\otimes(\hat{\bfV}_y^{(k)})^T\bfV_y^{(k)}\otimes(\hat{\bfV}_x^{(k)})^T\bfV_x^{(k)}\big)\text{vec}(\G^{(k)}).
\end{equation}

Projecting equation \eqref{eq: IMEX_k} onto the updated bases $\hat{\bfV}^{(k)}$, we get the third-order tensor linear equation
\begin{align}\label{eq: G_highorder}
\begin{split}
	&\Bigg\{\Big(\mathbf{I}_{\hat{r}^{(k)}\times \hat{r}^{(k)}}-a_{kk}\Delta t(\hat{\bfV}_z^{(k)})^T(\bfF_z\hat{\bfV}_z^{(k)})\Big)\otimes\mathbf{I}_{\hat{r}^{(k)}\times \hat{r}^{(k)}}\otimes\mathbf{I}_{\hat{r}^{(k)}\times \hat{r}^{(k)}}\\
	&\quad\quad+ \mathbf{I}_{\hat{r}^{(k)}\times \hat{r}^{(k)}}\otimes\Big(-a_{kk}\Delta t(\hat{\bfV}_y^{(k)})^T(\bfF_y\hat{\bfV}_y^{(k)})\Big)\otimes\mathbf{I}_{\hat{r}^{(k)}\times \hat{r}^{(k)}}\\
	&\quad\quad\quad+ \mathbf{I}_{\hat{r}^{(k)}\times \hat{r}^{(k)}}\otimes\mathbf{I}_{\hat{r}^{(k)}\times \hat{r}^{(k)}}\otimes\Big(-a_{kk}\Delta t(\hat{\bfV}_x^{(k)})^T(\bfF_x\hat{\bfV}_x^{(k)})\Big)\Bigg\}\text{vec}(\hat{\G}^{(k)}) = \text{vec}(\B^{(k-1)}),
\end{split}
\end{align}
where
\begin{align}\label{eq: IMEX_RAIL3D_Gk-1}
\begin{split}
	\B^{(k-1)} &= \G^n\times_1(\hat{\bfV}_x^{(k)})^T\bfV_x^n\times_2(\hat{\bfV}_y^{(k)})^T\bfV_y^n\times_3(\hat{\bfV}_z^{(k)})^T\bfV_z^n\\
&+ \Delta t\sum\limits_{\ell=1}^{k-1}a_{k\ell}\Big\{\G^{(\ell)}\times_1(\hat{\bfV}_x^{(k)})^T(\bfF_x\bfV_x^{(\ell)})\times_2(\hat{\bfV}_y^{(k)})^T\bfV_y^{(\ell)}\times_3(\hat{\bfV}_z^{(k)})^T\bfV_z^{(\ell)}\\
&\quad\quad\quad\qquad+ \G^{(\ell)}\times_1(\hat{\bfV}_x^{(k)})^T\bfV_x^{(\ell)}\times_2(\hat{\bfV}_y^{(k)})^T(\bfF_y\bfV_y^{(\ell)})\times_3(\hat{\bfV}_z^{(k)})^T\bfV_z^{(\ell)}\\
&\quad\quad\quad\qquad+ \G^{(\ell)}\times_1(\hat{\bfV}_x^{(k)})^T\bfV_x^{(\ell)}\times_2(\hat{\bfV}_y^{(k)})^T\bfV_y^{(\ell)}\times_3(\hat{\bfV}_z^{(k)})^T(\bfF_z\bfV_z^{(\ell)})\Big\}\\
&- \Delta t\sum\limits_{\ell=1}^{k}\tilde{a}_{k+1,\ell}\Big\{\G^{e,(\ell-1)}_1\times_1(\hat{\bfV}_x^{(k)})^T(\bfD_x\bfE_{1,x}^{(\ell-1)})\times_2(\hat{\bfV}_y^{(k)})^T\bfE_{1,y}^{(\ell-1)}\times_3(\hat{\bfV}_z^{(k)})^T\bfE_{1,z}^{(\ell-1)}\\
&\quad\quad\quad\quad\qquad+ \G^{e,(\ell-1)}_2\times_1(\hat{\bfV}_x^{(k)})^T\bfE_{2,x}^{(\ell-1)}\times_2(\hat{\bfV}_y^{(k)})^T(\bfD_y\bfE_{2,y}^{(\ell-1)})\times_3(\hat{\bfV}_z^{(k)})^T\bfE_{2,z}^{(\ell-1)}\\
&\quad\quad\quad\quad\qquad+ \G^{e,(\ell-1)}_3\times_1(\hat{\bfV}_x^{(k)})^T\bfE_{3,x}^{(\ell-1)}\times_2(\hat{\bfV}_y^{(k)})^T\bfE_{3,y}^{(\ell-1)}\times_3(\hat{\bfV}_z^{(k)})^T(\bfD_z\bfE_{3,z}^{(\ell-1)})\Big\}\\
&+ \Delta t\Big\{\G^{c,(\ell)}\times_1(\hat{\bfV}_x^{(k)})^T\bfC_x^{(\ell)}\times_2(\hat{\bfV}_y^{(k)})^T\bfC_y^{(\ell)}\times_3(\hat{\bfV}_z^{(k)})^T\bfC_z^{(\ell)}\Big\}.
\end{split}
\end{align}

The direct solver described in Corollary \ref{cor: 3Dsolver} is used to solve equation \eqref{eq: G_highorder} for $\hat{\G}^{(k)}$. According to some tolerance $\varepsilon$, the HOSVD (or a conservative truncation procedure) is then used to approximate the core tensor by a compressed Tucker tensor of multilinear rank $(r_x^{(k)},r_y^{(k)},r_z^{(k)})$,
\begin{equation}
	\hat{\G}^{(k)} \approx \G^{(k)}\times_1\bfG_x^{(k)}\times_2\bfG_y^{(k)}\times_3\bfG_z^{(k)}.
\end{equation}

Thus, the final updated solution at stage $k$ is described by the core tensor $\G^{(k)}$, and the one-dimensional bases/factor matrices $\bfV_x^{(k)} = \hat{\bfV}_x^{(k)}\bfG_x^{(k)}$, $\bfV_y^{(k)} = \hat{\bfV}_y^{(k)}\bfG_y^{(k)}$ and $\bfV_z^{(k)} = \hat{\bfV}_z^{(k)}\bfG_z^{(k)}$. Repeating this process for each subsequent stage, we eventually obtain the final updated solution $\U^{n+1}=\U^{(s)}=\G^{(s)}\times_1\bfV_x^{(s)}\times_2\bfV_y^{(s)}\times_3\bfV_z^{(s)}$. The 3d-RAIL method is summarized in Algorithm \ref{algo: highorderIMEX}.

\begin{rem}
Since the RAIL framework starts with the full discretization, the user has some flexibility in choosing the spatial and temporal discretizations. Preliminary results for the 2d-RAIL method in \cite{nakao2023thesis} used a second-order backward differentiation formula (BDF). The 3d-RAIL method can be extended to high-order BDF methods. One would just need to initialize the first steps using a sufficiently high-order scheme.
\end{rem}

\subsection{Stability and consistency}

Here, we discuss the stability and consistency of the 3d-RAIL scheme. We briefly discuss the structural differences and similarities between the 3d-RAIL and (implicit) Tucker-BUG algorithms, and how the first-order schemes share a similar error bound. A stability analysis is also provided for the first-order 3d-RAIL scheme. The consistency of the high-order scheme is then discussed without rigorous derivation, but we note that high-order accuracy was observed in the numerical experiments conducted for this paper. To the knowledge of the authors, only up to second-order error bounds have been rigorously derived within the class of \textit{implicit} DLR methods \cite{ceruti2024robust,kusch2025second}. Deriving error bounds for the high-order 3d-RAIL method is outside the scope of the current paper and remains a topic of ongoing work.

It was shown in \cite{nakao2025reduced} that the first-order 2d-RAIL and implicit 2d-BUG methods share the same error bound up to the tolerance of the reduced augmentation ($10^{-12}$). While their behavior at first-order is closely related in the 3d Tucker setting, there are structural differences between the RAIL and BUG methodologies in higher dimensions that do not apply to their 2d counterparts. Recall that 3d-RAIL fully discretizes in both space and time, before performing the projection onto a lower-dimensional subspace. This construction is more flexible for implicit time integrators compared to Tucker-BUG, which first projects the differential equation, resulting in time-continuous differential equations for the low-rank factors. Since the projection comes first, followed by splitting in the Tucker-BUG setting, a lower-order splitting error is incurred. Furthermore, Tucker-BUG freezes the core tensor along inactive modes when updating each low-rank factor matrix, which incurs an additional error. Whereas, RAIL-3D allows all matricized modes to evolve simultaneously. This flexibility, however, comes at the cost of increased computational complexity, as the K-steps typically incur a rank growth from $r$ to $r^2$. This structural distinction, rooted in the philosophy of DLR \textit{tensor approximation} \cite[\S 2]{koch2010dynamical}, enables a more faithful representation of the evolving solution and avoids the order-reducing effects induced by the sequential SVD-based derivation used in BUG. Moreover, unlike the Tucker-BUG method, the projection subspaces used in the 3d-RAIL scheme can be constructed in a more flexible manner at each RK stage.

\subsubsection{The first-order scheme}

To assess the consistency of the 3d-RAIL method, we begin by considering its first-order implicit version using backward Euler. This setting provides a natural point of comparison with the Tucker-BUG scheme introduced in \cite{ceruti2022rank}. The consistency of Tucker-BUG has been rigorously analyzed in \cite{ceruti2022rank}, where local and global error bounds were established. For completeness, we start by recalling the main local consistency result of the Tucker-BUG algorithm.

\begin{thm}[\cite{ceruti2022rank}]\label{thm: errorbound}
	Let the right-hand side of \eqref{eq: conv_diff_eqn} be Lipschitz continuous and uniformly bounded, with $L$ the Lipschitz constant and $B$ the uniform bound. Furthermore, assume that in a neighborhood of the exact solution, the right-hand side remains within $\sigma$-distance from its projection onto the tangent space of the rank-$r$ manifold. Let $\U_{\rm (BUG)}^{1}$ be the approximation obtained from the first-order implicit Tucker-BUG algorithm after a single time step $\Delta t>0$, assuming exact initial data $\U^0$. Then, the local error satisfies
	\begin{equation}
		\norm{\U(t^1)-\U_{\rm (BUG)}^{1}}_F \leq \Delta t(c_1\sigma + c_2\Delta t) + c_3\varepsilon,
	\end{equation}
	where the constants $c_i$ depend on $L$, $B$, and an upper bound on $\Delta t$. Here, $\varepsilon$ denotes the truncation tolerance used to compress the core tensor in the HOSVD.
\end{thm}

While the two algorithms adopt different design philosophies, particularly in how they update the low-rank factors, their behavior at first-order accuracy is closely related. Most notably, the K-steps in both schemes involve solving implicit linear systems to evolve each factor matrix. Despite the conceptual difference described earlier, the projection subspaces produced by both schemes \textit{at first-order} remain nearly equivalent due to the backward stability of the QR decomposition used in the basis construction. While the G-step in both methods employs the same type of core tensor reconstruction, the input factor matrices differ slightly between the two schemes, resulting in a discrepancy of order $\mathcal{O}(\Delta t^2)$ per time step. Since both algorithms target the same underlying dynamics and differ only in the construction of the updated low-rank factor matrices, the 3d-RAIL scheme reproduces the same leading-order behavior as Tucker-BUG. We remark that the $\mathcal{O}(10^{-12})$ error incurred by the reduced augmentation is within the scope of roundoff-level effects and does not impact the theoretical consistency of the scheme.

While consistency ensures that the scheme approximates the continuous problem correctly as the discretization parameters tend to zero, stability is crucial to guarantee that numerical errors do not grow uncontrollably over time. The following theorem establishes the stability of the first-order implicit 3d-RAIL scheme, ensuring controlled error growth under suitable assumptions.

\begin{thm}
	Assuming $\bfF_x$, $\bfF_y$, and $\bfF_z$ are symmetric and negative semi-definite, the first-order implicit RAIL-3D scheme before truncation is unconditionally stable, i.e.
	\[
	\| \hat \G^{n+1} \|_{\rm F} \leq \| \hat \G^{n} \|_{\rm F} \, .
	\] 
\end{thm}

\begin{proof}
	The proof extends the result originally provided in Theorem 2 of \cite{nakao2025reduced} from the matrix setting to the tensor setting. For clarity, matrix dimensions are omitted in the following. To simplify the notation, we introduce the symmetric matrices:
	\begin{align*}
		\bfA &:= \Delta t(\hat{\bfV}_z^{n+1})^T(\bfF_z\hat{\bfV}_z^{n+1}), \\
		\bfB &:= \Delta t(\hat{\bfV}_y^{n+1})^T(\bfF_y\hat{\bfV}_y^{n+1}), \\
		\bfC &:= \Delta t(\hat{\bfV}_x^{n+1})^T(\bfF_x\hat{\bfV}_x^{n+1}).
	\end{align*}
	The update of $\hat \G^{n+1}$ can be concisely expressed as the following tensor equation:
	\[
	\hat \G^{n+1} \times_1 (\mathbf{I} - \bfA) - \hat \G^{n+1} \times_2 \bfB - \hat \G^{n+1} \times_3 \bfC = \hat \G^{n} \, .
	\]
	Unfolding the above equation along the first mode, we obtain:
	\[
	(\mathbf{I} - \bfA) \hat \bfG^{n+1} - \hat \bfG^{n+1} (\mathbf{I} \otimes \bfB) - \hat \bfG^{n+1} (\bfC \otimes \mathbf{I}) = \hat \bfG^{n} \, .
	\]
	where the matrices \(\hat{\bfG}^{n+1}\) and \(\hat{\bfG}^{n}\) represent the matricization of the tensors \(\hat{\G}^{n+1}\) and \(\hat{\G}^{n}\) along the first mode, respectively. For simplicity, we omit the subscript indicating the matricization mode. Now, we diagonalize the matrices \(\bfA\), \(\bfB\), and \(\bfC\), which are symmetric and semi-definite; thus, they admit diagonal representations such as  
	\begin{align*}
		\bfD &:= \bfP^T \bfA \bfP,\\
		\bfE &:= \bfQ^T \bfB \bfQ,\\
		\bfF &:= \bfR^T \bfC \bfR,
	\end{align*} 
	where the factors $\bfP$, $\bfQ$, and $\bfR$ are orthonormal, e.g. $ \bfP \bfP^T = \bfP^T \bfP = \mathbf{I}$. Moreover, we have that 
	\[ (\bfQ \otimes \bfR) (\bfQ^T \otimes \bfR^T) = \mathbf{I} \otimes \mathbf{I} \, .\]
	We continue by diagonalizing the matrices appearing in the matrix equation. We begin by multiplying from the right with \(\bfP^T\) and \((\bfQ \otimes \bfR)\) on the left. This yields the intermediate matrix equation:  
	\[
	\bfP^T(\mathbf{I} - \bfA) \hat \bfG^{n+1} (\bfQ \otimes \bfR) - \bfP^T \hat \bfG^{n+1} (\mathbf{I} \otimes \bfB + \bfC \otimes \mathbf{I})(\bfQ \otimes \bfR) = \bfP^T \hat \bfG^{n} (\bfQ \otimes \bfR) \, .
	\]
	If we introduce the auxiliary variables 
	\[ 
	\widetilde \bfG^{n+1} :=  \bfP^T \hat \bfG^{n+1} (\bfQ \otimes \bfR), \quad
	\widetilde \bfG^{n} :=  \bfP^T \hat  \bfG^{n} (\bfQ \otimes \bfR) \, ,
	\]
	the matrix equation simplifies to
	\[
	\bfP^T(\mathbf{I} - \bfA) \bfP \widetilde \bfG^{n+1} - \widetilde \bfG^{n+1} (\bfQ^T \otimes \bfR^T) (\mathbf{I} \otimes \bfB + \bfC \otimes \mathbf{I})(\bfQ \otimes \bfR) = \widetilde \bfG^n \, .
	\]
	Thus, since most of the factors in the expression above can be diagonalized, we obtain  
	\[
	(\mathbf{I} - \bfD) \widetilde \bfG^{n+1} - \widetilde \bfG^{n+1} (\mathbf{I} \otimes \bfE + \bfF \otimes \mathbf{I}) = \widetilde \bfG^n \, .
	\]
	It follows that 
	\[
	\widetilde G^{n+1}_{ij} = \alpha_{ij} \widetilde G^{n}_{ij}, \quad \text{with} \quad 
	\alpha_{ij} = \frac{1}{1 - D_{ii} - E_{jj} - F_{jj}} \, .
	\]
	Since the eigenvalues are strictly negative, the amplification factor \(\alpha_{ij}\) is less than or equal to 1. Thus,  
	\[
	\|\hat  \G^{n+1} \|_{\rm F} = \| \hat \bfG^{n+1} \|_{\rm F} = \| \widetilde \bfG^{n+1} \|_{\rm F} \leq \| \widetilde \bfG^{n} \|_{\rm F} \, ,
	\]
	The conclusion follows by recalling that, analogously, $ \| \widetilde \bfG^{n} \|_{\rm F} = \| \hat \G^n \|_{\rm F}$.
\end{proof}

At a more abstract level, this stability result can be attributed to the structural design of the algorithm: the dynamics are first discretized in time using an implicit scheme, and only then is a low-rank approximation applied to the resulting intermediate stages. This ordering preserves the inherent stability properties of the underlying time integrator. Moreover, the structure-preserving nature of the BUG-type low-rank updates, on which the RAIL formulation is built, ensures that these favorable stability characteristics are inherited.

\subsubsection{Addressing high-order accuracy}
While first-order consistency of the 3d-RAIL scheme can be explored through comparisons with Tucker-BUG, achieving high-order accuracy requires a more careful evaluation. When using high-order RK methods, our approach builds on the idea of interpreting each RK stage as a retraction step onto the low-rank manifold.

Unlike the DLR formulation which evolves the low-rank factors continuously in time, the 3d-RAIL perspective changes since the dynamics are evolved discretely in time according to the RK method. As a result, the standard viewpoints upon which the theory is based on in the DLR literature does not directly apply here. Thus, a new retraction-based formulation becomes necessary to preserve accuracy and structure across multiple stages. For a diagonally implicit Runge–Kutta (DIRK) method, the $i$-th intermediate stage satisfies
\begin{equation}
	\mathcal{Y}^{(i)} = \mathcal{U}^n + \Delta t \sum_{j=1}^{i} a_{ij} F\left( \mathcal{U}^{(j)} \right),
\end{equation}
where $\mathcal{Y}^{(i)}$ denotes the intermediate value that would be computed in a classical implementation, and $\mathcal{U}^{(j)}$ represents its low-rank surrogate at stage $j$. In the 3d-RAIL scheme, this is approximated via projection onto a tangent space:
\begin{equation}
	\mathcal{Y}^{(i)}  
	\approx  \rm P_{\mathcal{T}_{\widehat{\mathcal{U}}^{(i-1)}}} \mathcal{Y}^{(i)} \, .
\end{equation}
Here, $\widehat{\mathcal{U}}^{(i-1)}$ is an augmented representation used to define the tangent space, typically constructed by enriching the previous low-rank states $\mathcal{U}^{(i-1)}$ with a first-order prediction of the next stage. This construction effectively discards the normal component of the update, an approximation justified by a standard assumption in the DLR literature: for sufficiently small time-steps, the normal component remains negligible. As a result, a retraction can be applied directly to the projected candidate to map it back to the low-rank manifold. Thus, in the 3d-RAIL framework the update is interpreted geometrically as
\begin{equation}
	\mathcal{U}^{(i)} = \mathcal{R} \left( \rm P_{\mathcal{T}_{\widehat{\mathcal{U}}^{(i-1)}}} \mathcal{Y}^{(i)} \right) \, .
\end{equation}
This strategy is similar in spirit to the projection-based schemes proposed for time-dependent low-rank matrices in~\cite{kieri2019projection}, where the tangent space is dynamically adapted using \emph{explicit} RK methods. In the matrix setting, it has been recently shown that BUG-type updates can be interpreted as high-order retractions compatible with the underlying RK structure~\cite{seguin2024low}, ensuring that the correct order of accuracy is preserved at each intermediate stage. This result suggests a theoretical justification for the use of BUG-style procedures within multistage integrators. While a rigorous analysis of BUG-type retractions in the Tucker tensor setting remains an open challenge, the present 3d-RAIL construction is designed to preserve the same principles. Numerical evidence suggests that the 3d-RAIL scheme successfully maintains high-order accuracy when combined with implicit RK methods. This favorable behavior is largely due to the structural design of the algorithm, particularly the use of reduced augmentation, which enriches the projection subspaces at each stage by incorporating both previous \textit{and} predicted information.

\subsection{Computational complexity}
We briefly comment on the computational complexity of the 3d-RAIL algorithm. The proposed algorithm offers an efficient way to solve three-dimensional diffusion and convection-diffusion equations that exhibit low-rank solutions. However, we highlight the steps in the algorithm that will observe increased cost if the rank $r$ starts to grow too large. Each augmented matrix \eqref{eq: redaug_RK} that needs to be reduced is size $N\times r(k+1)$, where $k=1,2,...,s$ is the stage of the RK method. Hence, the computational cost of the reduced augmentations will increase with higher-order RK methods as more stages are required.

The dominant computational cost of the K steps comes from solving matrix Sylvester equations. Without loss of generality, the coefficient matrices involved in the K$_1$ equation are $\mathbf{I}-a_{kk}\Delta t\mathbf{F}_x$ of size $N\times N$ and $a_{kk}\Delta t(\bfF_z\bfV_z^{\star,(k)})^T\bfV_z^{\star,(k)}\oplus(\bfF_y\bfV_y^{\star,(k)})^T\bfV_y^{\star,(k)}$ of size $r^2\times r^2$. As such, most standard Sylvester solvers, e.g., \cite{bartels1972,golub1979hessenberg}, will have a computational cost dominated by $\sim N^3$ flops, assuming $r\ll N$. One could diagonalize this coefficient matrix and solve transformed Sylvester equations, or exploit potential sparsity of $\mathbf{F}_x$ using an iterative scheme to reduce this computation. However, if the rank $r$ starts significantly increasing, the other flop counts on the order of $\sim N^2r^2$, $\sim Nr^4$ and $\sim r^6$ will noticeably affect the computational cost. Assuming $r\ll N$, it's also worth noting that the reduced QR factorization of $\mathbf{K}_1$ will cost $\sim Nr^4$ flops.

Similarly, the direct solver used in the G step boils down to solving several Sylvester equations \eqref{eq: SimonciniSolver_sylveqn}. Since we use the direct solver to solve the G equation \eqref{eq: G_highorder} in which all the coefficient matrices are of size $r\times r$, the computational cost for solving each Sylvester equation will be dominated by $\sim r^3$ flops. Doing this $r$ times for the direct solver will result in the G step costing $\sim r^4$ flops.

Recall that the rank of the flow field can increase the cost, as seen in equation \eqref{eq: flowfield_tuckerformat}. Assuming the (multilinear) rank of the flow field is $r'$ and the (multilinear) rank of the solution is $r$, the size of the resulting core tensor $\G^e_i$ is $rr'\times rr'\times rr'$; the mode-$n$ matricization is size $rr'\times (rr')^2$. As a result, the computational cost of subsequent matrix multiplications greatly increases if the rank of the flow field is relatively large ($r^2\geq N$). If more general low-rank flow fields are considered, then further compression techniques might be needed for the resulting Tucker decomposition in order to maintain efficiency; see \cite{kressner2017recompression}. In our numerical tests, the rank of the solution appears to be independent of the mesh, meaning that the computational savings of the proposed method on low-rank problems becomes relatively better with finer meshes.

\begin{rem}
The non-conservative HOSVD conserves mass up to the truncation tolerance, but higher moments, especially the second moment for energy, are usually not well conserved. While the conservative truncation procedure described in the Appendix has the same computational complexity as the non-conservative HOSVD truncation, the overall cost is greater because the moments must be computed, and two HOSVD truncations must be performed. In that sense, there is a tradeoff between runtime and conservation. In the numerical results that we present in this paper, if a conservative truncation was used, it was used to truncate at \textit{each} stage. Alternatively, one could use the cheaper non-conservative HOSVD at all stages except the final one, at which point a conservative truncation procedure could be used to correct the conservation of the moments. We opted for the prior since we didn't want the accumulated conservation loss to grow too much over many stages, although we observed nearly identical results using either choice.
\end{rem}

\begin{rem}\label{rem: lomac_increasedrank}
In principle, the LoMaC truncation should increase the rank by the number of moments one wants to conserve. For instance, if one desires mass, momentum, and energy conservation, then the multilinear rank should increase by $(3,3,3)$. Naturally, the observed rank of the solution might be slightly larger.
\end{rem}


\section{Numerical experiments}
We now test the 3d-RAIL algorithm on various benchmark problems. Rank plots show how well the scheme captures low-rank solution structures, and $L^1$ error plots demonstrate the high-order accuracy in time. We assume a uniform mesh in space with $N=N_x=N_y=N_z$ gridpoints in each dimension. Given the uniform mesh, spectral collocation methods discretize the one-dimensional spatial derivatives, with the differentiation matrices found in \cite{Trefethen2000}. As such, we expect the temporal error to dominate. Although we assume periodic boundary conditions, other boundary conditions (and hence other differentiation matrices) could be used. For solutions that decay at infinity, the computational domain is made large enough for sufficient smoothness at the boundary for spectral methods to be used assuming periodic boundary conditions. The Butcher tableaus for the IMEX RK schemes used in our tests are listed in the Appendix. Depending on the model, we either use the HOSVD or LoMaC truncation procedure.

As in the 2d-RAIL scheme \cite{nakao2025reduced}, the time-stepping size for the three-dimensional RAIL scheme appears to satisfy a CFL condition when using IMEX discretizations due to the explicit treatment of the transport terms. It has been shown, at least in the discontinuous Galerkin framework, that when solving advection-diffusion problems with IMEX discretizations, the time-stepping size must be upper-bounded by a constant dependent on the ratio of the diffusion and the square of the advection coefficients \cite{wang2015stability}. Although our proposed scheme is in the finite difference framework, we observed similar restrictions in our numerical tests. As such, we define the time-stepping size by
\begin{equation}\label{eq: dt_restriction}
	\Delta t = \frac{\lambda}{\frac{\max{|f_1'(u)|}}{\Delta x} + \frac{\max{|f_2'(u)|}}{\Delta y} + \frac{\max{|f_3'(u)|}}{\Delta z}},
\end{equation}
where $\lambda>0$ is the CFL scaling factor. We define the $L^1$ error as
\begin{equation}
	\norm{u-u_{\text{exact}}}_1\coloneqq \Delta x\Delta y\Delta z\sum\limits_{i=1}^{N_x}\sum\limits_{j=1}^{N_y}\sum\limits_{k=1}^{N_z}{|u_{ijk}-u_{\text{exact},ijk}|},
\end{equation}
where we do not scale by measure of the domain. Lastly, we set the low-rank tolerance $\varepsilon$ between $10^{-5}$ and $10^{-8}$. In practice, this tolerance can be made larger assuming it is smaller than the local truncation error. The LoMaC truncation procedure scales the solution by a weight function that is assumed to be a Gaussian with sufficient decay; see \cite{guo2024local} and Appendix \ref{app: lomac}. We define the weight function to be $w(\mathbf{x})=\text{Exp}(-s|\mathbf{x}|^2)$, for $s>0$ large enough.

\subsection[]{Advection-diffusion with constant coefficients}

\begin{equation}\label{eq: tests_advdiff_equation}
u_t +u_x+u_y+u_z = d(u_{xx}+u_{yy}+u_{zz}),\qquad x,y,z\in(-\pi,\pi)
\end{equation}
where $d=1/6$. We use the first two Fourier modes, with the exact solution
\begin{equation}\label{eq: advdiff_IC}
	u(x,y,z,t) = 1 + \sum\limits_{k=1}^{2}{e^{-3dk^2t}\sin{(k(x-t))}\sin{(k(y-t))}\sin{(k(z-t))}}.
\end{equation}
As seen in Figure \ref{fig: advdiff_accuracytest}, the expected accuracies are observed for the RAIL scheme when using IMEX111, IMEX222 and IMEX443 with mass conservative LoMaC truncation (we $s=4$). We used a mesh size $N=100$, tolerance $\varepsilon=10^{-6}$, final time $T_f=0.5$, and $\lambda$ ranging from 0.2 to 1. Despite observing the expected accuracy, the $L^1$ error for the first-order scheme (and even the second-order scheme) is quite large. However, recall that we do not scale the $L^1$ error by the measure of the domain, which in this case would be $|\mathbf{\Omega}|=(2\pi)^3$; scaling by the measure of the domain would provide a better comparison against the $L^{\infty}$ error which is not as large.

\begin{figure}[t!]
\centering
\begin{minipage}[b]{0.48\linewidth}
	\centering
	\includegraphics[width=0.85\textwidth]{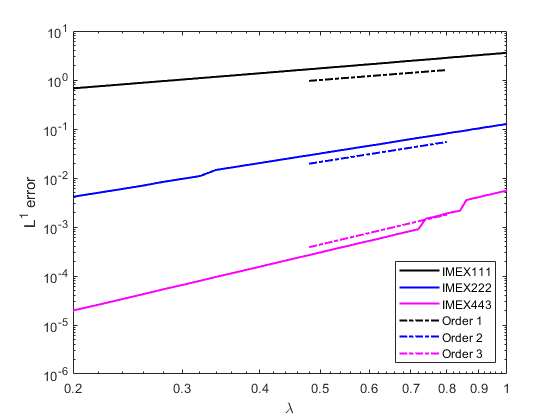}
	\caption{Error plot for \eqref{eq: tests_advdiff_equation}.}
	\label{fig: advdiff_accuracytest}
\end{minipage}
\begin{minipage}[b]{0.02\linewidth}
\ \\
\end{minipage}
\begin{minipage}[b]{0.48\linewidth}
	\centering
	\includegraphics[width=0.85\textwidth]{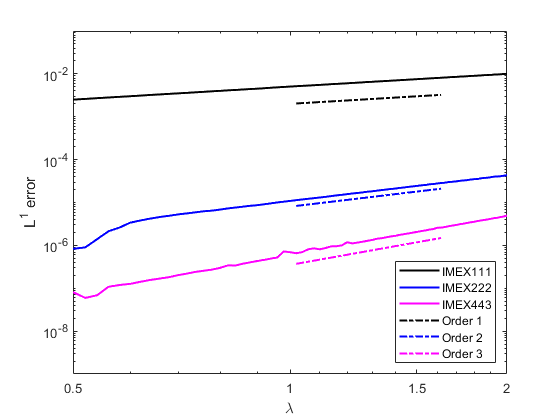}
	\caption{Error plot for \eqref{eq: tests_rotation1}.}
	\label{fig: rotation_accuracytest}
\end{minipage}
\end{figure}

\subsection[]{Rigid body rotation with diffusion, about $\hat{z}$}

\begin{equation}\label{eq: tests_rotation1}
u_t - yu_x+xu_y = d(u_{xx}+u_{yy}+u_{zz}) + c(x,y,z,t),\qquad x,y,z\in(-2\pi,2\pi)
\end{equation}
where the flow field describes rotation about the vector $\hat{z}$. To test the accuracy of the scheme, we use the manufactured solution $u(x,y,z,t)=\text{exp}(-(x^2+2y^2+3z^2+3dt))$ with $d=1/3$, for which the source term $c(x,y,z,t)$ offsets the rotation and is
\begin{equation}\label{eq: rotation_sourceterm}
	c(x,y,z,t) = e^{-(x^2+2y^2+3z^2+3dt)}\Big(-2xy-d(-9+4x^2+16y^2+36z^2)\Big).
\end{equation}
As seen in Figure \ref{fig: rotation_accuracytest}, the expected accuracies are observed for the RAIL scheme when using IMEX111, IMEX222 and IMEX443 with non-conservative HOSVD truncation; due to the source term, the mass is not conserved. We used a mesh size $N=100$, tolerance $\varepsilon=10^{-8}$, final time $T_f=0.5$, and $\lambda$ ranging from 0.5 to 2. When a low-rank source term is involved, we must express it in a Tucker tensor format. By inspection, it is straightforward for one to write down a Tucker decomposition of \eqref{eq: rotation_sourceterm}.

To test the rank of the solution, we set $d=1/12$, $c(x,y,z,t)=0$, double the speed of the rotation,
\begin{equation}\label{eq: tests_rotation2}
u_t - 2yu_x+2xu_y = \frac{1}{12}(u_{xx}+u_{yy}+u_{zz}),\qquad x,y,z\in(-2\pi,2\pi)
\end{equation}
and set the initial condition to $u_0(x,y,z)=\text{exp}(-(x^2+9y^2+z^2))$. Since the mass is conserved for this problem, we use the mass conservative LoMaC truncation (we set $s=1$). The solution rotates counterclockwise about the positive $z$-axis while slowly diffusing. Theoretically, the exact multilinear rank should be $(1,1,1)$ as the solution (re)aligns with the axes at $t=0$, $t=\pi/4$, $t=\pi/2$, $t=3\pi/4$ and so forth. The rank in $x$ and $y$ should increase and decrease in between these time stamps as the solution rotates about $\hat{z}$, whereas the rank in $z$ should remain one since the rotation is only occuring in the $xy$-plane. As seen in Figure \ref{fig: rotation_ranktest1}, the RAIL scheme captures this behavior in all three plots, although the multilinear rank doesn't quite reach $(1,1,1)$ at $t=0,\pi/2,\pi/2,3\pi/4$ as per Remark \ref{rem: lomac_increasedrank}. The decrease in the magnitude of the ``humps" due to the slow diffusion is captured in all three plots. For the rank plots, we used a mesh size $N=100$, tolerance $\varepsilon=10^{-6}$, and $\lambda=0.45$. Mass conservation was observed to machine precision, as seen in Figure \ref{fig: rotation_mass1}.

\begin{figure}[t!]
\centering
\begin{minipage}[b]{0.32\linewidth}
	\centering
	\includegraphics[width=\textwidth]{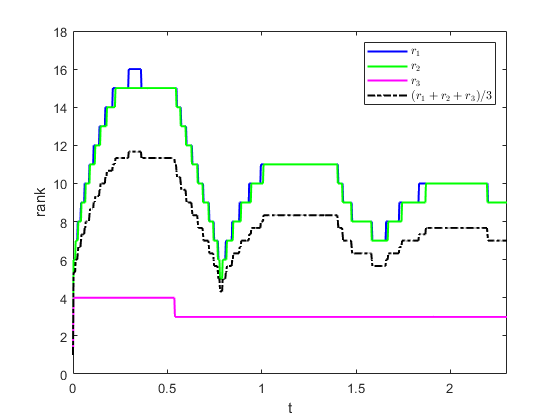}
\end{minipage}
\begin{minipage}[b]{0.32\linewidth}
	\centering
	\includegraphics[width=\textwidth]{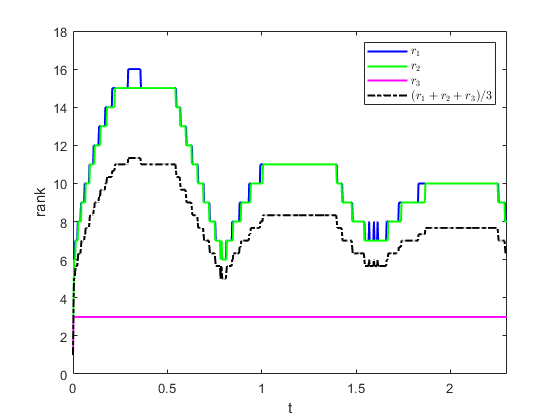}
\end{minipage}
\begin{minipage}[b]{0.32\linewidth}
	\centering
	\includegraphics[width=\textwidth]{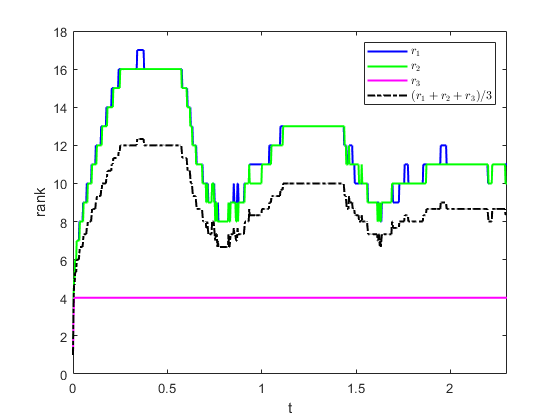}
\end{minipage}
\caption{The multilinear rank and average rank of the solution to \eqref{eq: tests_rotation2} with initial condition $u_0(x,y,z)=\text{exp}(-(x^2+9y^2+z^2))$ using IMEX111 \textit{(left)}, IMEX222 \textit{(middle)} and IMEX443 \textit{(right)}.}
\label{fig: rotation_ranktest1}
\end{figure}

\begin{figure}[t!]
\centering
\begin{minipage}[b]{0.48\linewidth}
	\centering
	\includegraphics[width=0.85\textwidth]{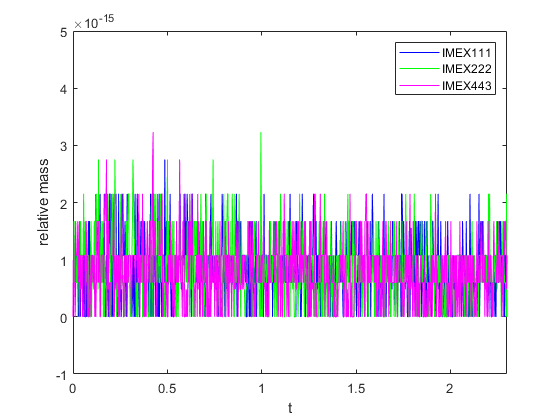}
	\caption{Relative change in mass for \eqref{eq: tests_rotation2}.}
	\label{fig: rotation_mass1}
\end{minipage}
\begin{minipage}[b]{0.02\linewidth}
\ \\
\end{minipage}
\begin{minipage}[b]{0.48\linewidth}
	\centering
	\includegraphics[width=0.85\textwidth]{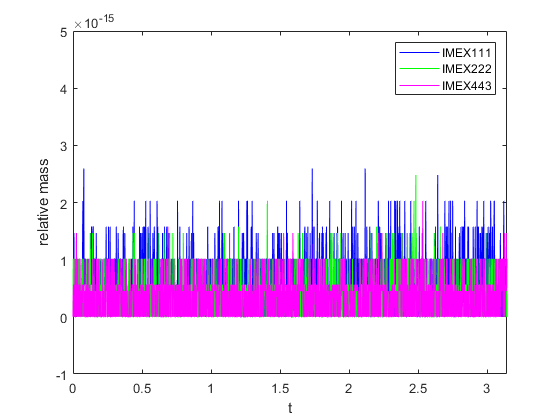}
	\caption{Relative change in mass for \eqref{eq: tests_rotation3}.}
	\label{fig: rotation_mass2}
\end{minipage}
\end{figure}

In our experiments, we observed that the solution only diffused and did not rotate when using large truncation tolerances and/or very small time-stepping sizes with the non-conservative HOSVD truncation. This behavior was especially seen for the first-order IMEX111 time discretization, e.g., with $\varepsilon=10^{-6}$ and $\lambda=0.25$; decreasing the tolerance corrected this, as did increasing the time-stepping size. This observation was also seen with IMEX222, although only for extremely large tolerances, e.g., $\varepsilon=10^{-1}$, and extremely small time-stepping sizes. This behavior is known to occur with (low-order) BUG methods since the rotation is far from the tangent space and thus requires either a very small truncation tolerance or a large time-stepping size so that the rotation dynamics are more dominant; see Example 4.1 in \cite{appelo2025robust}.

\subsection[]{Rigid body rotation with diffusion, about $\hat{x}+\hat{y}+\hat{z}$}
\begin{equation}\label{eq: tests_rotation3}
u_t + (-y+z)u_x+(x-z)u_y+(-x+y)u_z = d(u_{xx}+u_{yy}+u_{zz}),\qquad x,y,z\in(-2\pi,2\pi)
\end{equation}
where $d=1/12$ and the flow field describes rotation about the vector $\hat{x}+\hat{y}+\hat{z}$. As with equation \eqref{eq: tests_rotation2}, we use the mass conservative LoMaC truncation (we set $s=1$). Here, we test if the RAIL scheme accurately captures and maintains low-rank structure in solutions for which rotation occurs in all directions. We set the initial condition as $u_0(x,y,z)=\text{exp}(-2((x-\pi/2)^2+(y+\pi/2)^2+z^2))$. This is a Gaussian distribution centered at the point $(\pi/2,-\pi/2,0)$ that rotates counterclockwise about the vector $\hat{x}+\hat{y}+\hat{z}$ while slowly diffusing. As expected, Figure \ref{fig: rotation_ranktest2} shows the RAIL scheme maintains the low-rank structure in the solution for all three time discretizations, with IMEX443 performing the best. The ranks are slightly higher than one might expect, as per Remark \ref{rem: lomac_increasedrank}. We used a mesh size $N=100$, tolerance $\varepsilon=10^{-6}$, and $\lambda=0.5$. Mass conservation was observed to machine precision, as seen in Figure \ref{fig: rotation_mass2}.
\begin{figure}[t!]
\centering
\begin{minipage}[b]{0.32\linewidth}
	\centering
	\includegraphics[width=\textwidth]{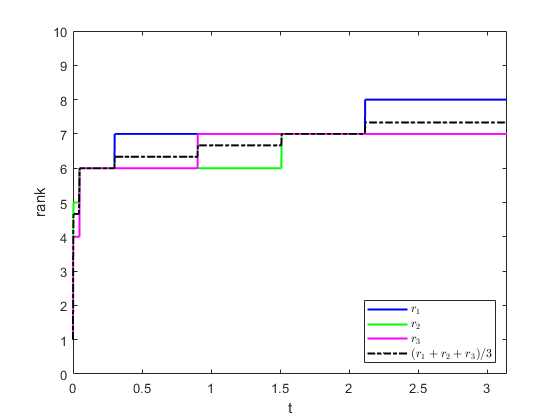}
\end{minipage}
\begin{minipage}[b]{0.32\linewidth}
	\centering
	\includegraphics[width=\textwidth]{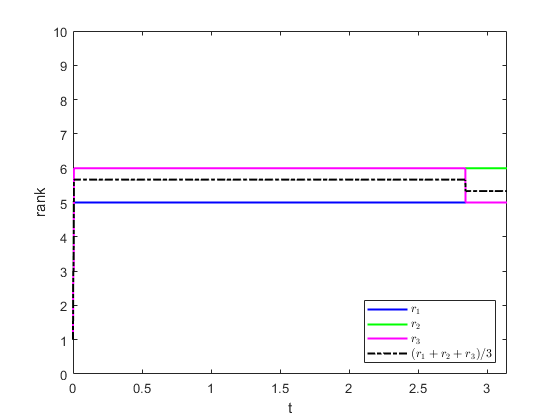}
\end{minipage}
\begin{minipage}[b]{0.32\linewidth}
	\centering
	\includegraphics[width=\textwidth]{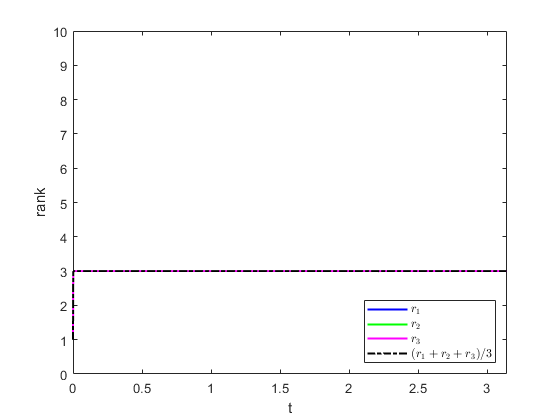}
\end{minipage}
\caption{The multilinear rank and average rank of the solution to \eqref{eq: tests_rotation3} with initial condition $u_0(x,y,z)=\text{exp}(-2((x-\pi/2)^2+(y+\pi/2)^2+z^2))$ using IMEX111 \textit{(left)}, IMEX222 \textit{(middle)} and IMEX443 \textit{(right)}.}
\label{fig: rotation_ranktest2}
\end{figure}

\subsection[]{Rigid body rotation with diffusion and time-dependent flow field, about $\hat{z}$}
\begin{equation}\label{eq: tests_rotationtimedependent}
u_t - tyu_x+txu_y = d(u_{xx}+u_{yy}+u_{zz}) + c(x,y,z,t),\qquad x,y,z\in(-2\pi,2\pi)
\end{equation}
where the flow field describes rotation about the vector $\hat{z}$ in which the speed starts at zero and increases linearly with $t$. To test the accuracy of the scheme, we use the manufactured solution $u(x,y,z,t)=\text{exp}(-(x^2+2y^2+3z^2+3dt))$ with $d=1/3$, for which the source term $c(x,y,z,t)$ is
\begin{equation}\label{eq: rotationtimedependent_sourceterm}
	c(x,y,z,t) = e^{-(x^2+2y^2+3z^2+3dt)}\Big(-2txy-d(-9+4x^2+16y^2+36z^2)\Big).
\end{equation}
As seen in Figure \ref{fig: rotationtimedependent_accuracytest}, the expected accuracies are observed for the RAIL scheme when using IMEX111, IMEX222 and IMEX443 with non-conservative HOSVD truncation; due to the source term, the mass is not conserved. We used a mesh size $N=100$, tolerance $\varepsilon=10^{-8}$, final time $T_f=0.5$, and $\lambda$ ranging from 0.5 to 2. As with the time-independent flow field case, it is straightforward to write down a Tucker decomposition of \eqref{eq: rotationtimedependent_sourceterm}.

\begin{figure}[t!]
\centering
\begin{minipage}[b]{0.48\linewidth}
	\centering
	\includegraphics[width=0.85\textwidth]{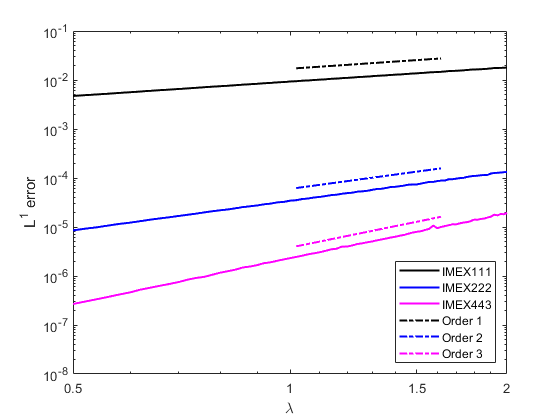}
	\caption{Error plot for \eqref{eq: tests_rotationtimedependent}.}
	\label{fig: rotationtimedependent_accuracytest}
\end{minipage}
\begin{minipage}[b]{0.02\linewidth}
\ \\
\end{minipage}
\begin{minipage}[b]{0.48\linewidth}
	\centering
	\includegraphics[width=0.85\textwidth]{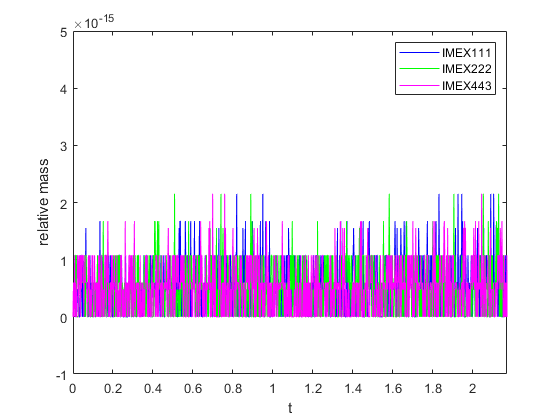}
	\caption{Relative change in mass for \eqref{eq: tests_rotationtimedependent2}.}
	\label{fig: rotation_mass3}
\end{minipage}
\end{figure}

To test the rank of the solution, we set $d=1/12$, $c(x,y,z,t)=0$, double the speed of the rotation,
\begin{equation}\label{eq: tests_rotationtimedependent2}
u_t - 2tyu_x+2txu_y = \frac{1}{12}(u_{xx}+u_{yy}+u_{zz}),\qquad x,y,z\in(-2\pi,2\pi)
\end{equation}
and set the initial condition to $u_0(x,y,z)=\text{exp}(-(x^2+9y^2+z^2))$. Since mass is conserved for this problem, we use the mass conservative LoMaC truncation (we set $s=2$). This is essentially the same rank test as was done for equation \eqref{eq: tests_rotation2}, but the speed has been scaled linear in $t$ so that the solution will (re)align with the axes at times $t=\sqrt{\pi/1},\sqrt{\pi},\sqrt{3\pi/2}$, and so forth. As seen in Figure \ref{fig: rotationtimedependent_ranktest}, the RAIL scheme captures this behavior with IMEX111, IMEX222 and IMEX443, although the multilinear rank doesn't quite reach $(1,1,1)$ at $t=\sqrt{\pi/1},\sqrt{\pi},\sqrt{3\pi/2}$, as per Remark \ref{rem: lomac_increasedrank}. The decrease in the magnitude of the ``humps" due to the slow diffusion is captured in all three plots, but IMEX222 and IMEX443 do a better job capturing the return to low-rank. For the rank plots, we used a mesh size $N=100$, tolerance $\varepsilon=10^{-6}$, and $\lambda=0.9$. Mass conservation was observed to machine precision, as seen in Figure \ref{fig: rotation_mass3}.

\begin{figure}[t!]
\centering
\begin{minipage}[b]{0.32\linewidth}
	\centering
	\includegraphics[width=\textwidth]{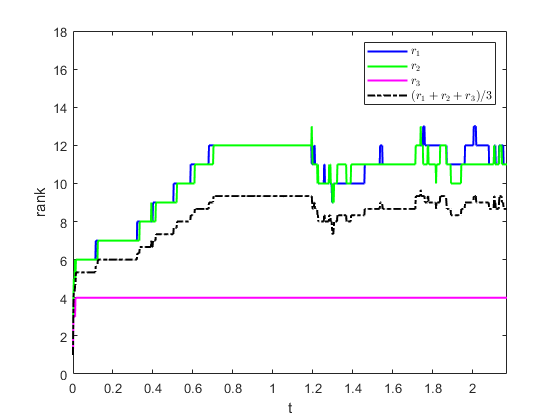}
\end{minipage}
\begin{minipage}[b]{0.32\linewidth}
	\centering
	\includegraphics[width=\textwidth]{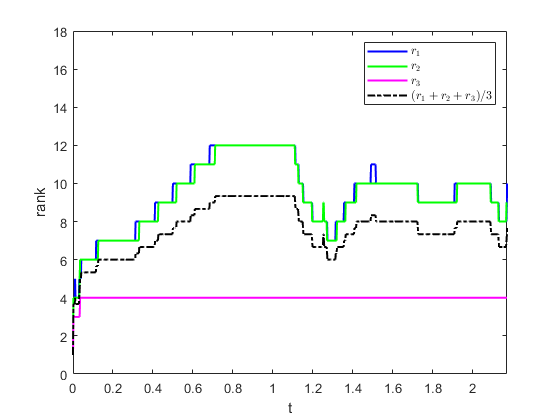}
\end{minipage}
\begin{minipage}[b]{0.32\linewidth}
	\centering
	\includegraphics[width=\textwidth]{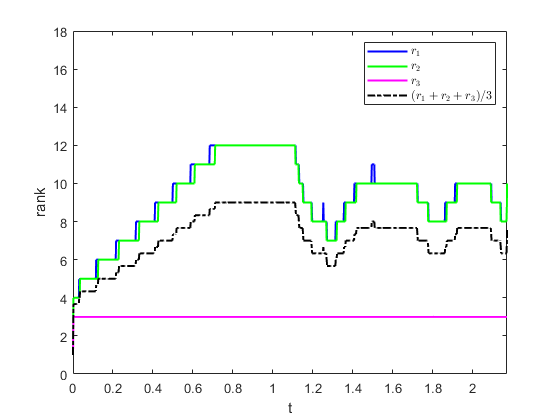}
\end{minipage}
\caption{The multilinear rank and average rank of the solution to \eqref{eq: tests_rotationtimedependent2} with initial condition $u_0(x,y,z)=\text{exp}(-(x^2+9y^2+z^2))$ using IMEX111 \textit{(left)}, IMEX222 \textit{(middle)} and IMEX443 \textit{(right)}.}
\label{fig: rotationtimedependent_ranktest}
\end{figure}

\subsection[]{Dougherty-Fokker-Planck equation}
\begin{equation}\label{eq: tests_DFP}
f_t - ((v_x-\overline{v}_x)f)_{v_x}-((v_y-\overline{v}_y)f)_{v_y}-((v_z-\overline{v}_z)f)_{v_z} = D(f_{v_xv_x}+f_{v_yv_y}+f_{v_zv_z}),\quad v_x,v_y,v_z\in(-8,8)
\end{equation}

Without loss of generality, we set the thermal speed to $v_{th}=\sqrt{2D}=\sqrt{2RT}=1$. We desire conservation of mass, momentum, and energy, obtained by taking the first few moments of the distribution function,
\begin{equation}
    \left(n,n\mathbf{u},E\right)^T = \int_{\mathbb{R}^3}{\left(1,\mathbf{v},\frac{|\mathbf{v}|^2}{2}\right)^Tfd^3v},
\end{equation}
where $E=(n|\mathbf{u}|^2+3nRT)/2$ is the energy and is related to the temperature $T=\frac{1}{3Rn}\int_{\mathbb{R}^3}{|\mathbf{v}-\mathbf{u}|^2fd^3v}$. Given the number density $n$, bulk velocity $\mathbf{u}$, and temperature $T$, the equilibrium solution is the Maxwellian distribution function
\begin{equation}\label{eq: DFP_relaxation}
	f_M^{\infty}(\mathbf{v},t=0) = \frac{n}{(2\pi RT)^{3/2}}\text{exp}\left(-\frac{|\mathbf{v}-\mathbf{u}|^2}{2RT}\right).
\end{equation}
Relaxation of the system is tested using the initial distribution function
\begin{equation}\label{eq: DFP_IC}
	f_0(\mathbf{v}) = f_{M1}(\mathbf{v}) + f_{M2}(\mathbf{v}),
\end{equation}
where $f_{M1}$ and $f_{M2}$ are two randomly generated Maxwellians such that the total number density, bulk velocity, and temperature are $n=\pi^{3/2}$, $\mathbf{u}=\mathbf{0}$, and $T=3$, respectively. As such, the solution will relax to Maxwellian \eqref{eq: DFP_relaxation} as $t\rightarrow\infty$. The number densities, bulk velocities, and temperatures that define the Maxwellians $f_{M1}$ and $f_{M2}$ are listed in Table \ref{table: twomaxwellians}.
\begin{table}[h!] 
\begin{center}
\label{}
\begin{tabular}{|c|c|c|} 
\hline 
&$f_{M1}$&$f_{M2}$\\
\hline
$n$&0.8037121822811545&4.764615814550553\\
\hline
$u_x$&-0.3403147128006618&0.05740548475117823\\
\hline
$u_y$&0&0\\
\hline
$u_z$&0&0\\
\hline
$T$&0.1033754314349305&3.442950196134546\\
\hline
\end{tabular} 
\caption{$n=\pi^{3/2}$, $\overline{\mathbf{v}}=\mathbf{0}$, and $T=3$.}
\label{table: twomaxwellians}
\end{center} 
\end{table}

The initial condition \eqref{eq: DFP_IC} is rank-two, and the equilibrium solution \eqref{eq: DFP_relaxation} is rank-one. The rank of the solution will immediately increase before quickly decreasing as the solution relaxes. Figure \ref{fig: DFP_ranktest} shows the RAIL scheme captures this behavior with all three time discretizations. Due to their smaller \textit{temporal} error, higher-order schemes do a better job capturing rapid changes of the solution in time, e.g., the fast rank decrease; one could also use a smaller time-stepping size to same same effect. We used a mesh size $N=100$, tolerance $\varepsilon=10^{-6}$, and $\lambda=0.5$. The LoMaC truncation procedure was used to conserve mass, momentum, and energy (we set $s=1$); see Figure \ref{fig: DFP_moments}.

Last, we show $L^1$ convergence $\norm{f-f^{\infty}_M}_1$ and Kullback relative entropy dissipation $\int_{\mathbb{R}^3}{f\log{(f/f^{\infty}_M)}d^3v}$ of the solution. To ensure the correct steady-state Maxwellian distribution is used, we leverage the quadrature corrected moment (QCM) procedure in \cite{taitano_jcp_2017_ep_rfp} to compute the Maxwellian distribution \eqref{eq: DFP_relaxation} whose \textit{discrete} moments match the discrete moments of the initial distribution \eqref{eq: DFP_IC}. We used a mesh $N=100$, tolerance $\varepsilon=10^{-4}$, and $\lambda=0.9$. Figure \ref{fig: DFP_L1decay} shows good equilibrium preservation, with the solution converging to the corrected Maxwellian distribution \eqref{eq: DFP_relaxation} up to $\mathcal{O}(10^{-8})$ despite using a vary large truncation tolerance of $10^{-4}$; the TensorLab library performs many operations in single precision for computational efficiency. In addition, the discrete Kullback relative entropy decreases slightly better, down to $\mathcal{O}(10^{-11})$.

\begin{figure}[t!]
\centering
\begin{minipage}[b]{0.32\linewidth}
	\centering
	\includegraphics[width=\textwidth]{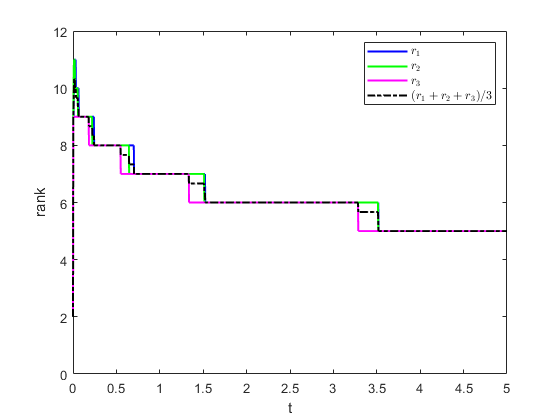}
\end{minipage}
\begin{minipage}[b]{0.32\linewidth}
	\centering
	\includegraphics[width=\textwidth]{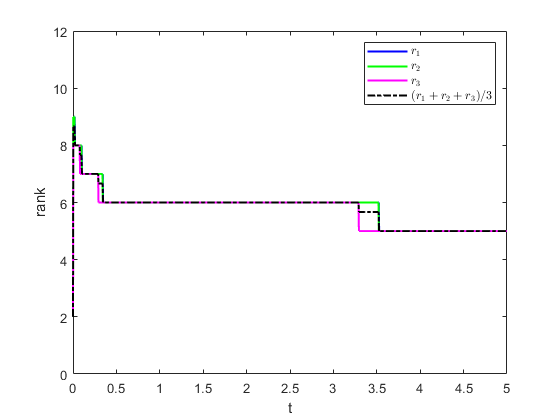}
\end{minipage}
\begin{minipage}[b]{0.32\linewidth}
	\centering
	\includegraphics[width=\textwidth]{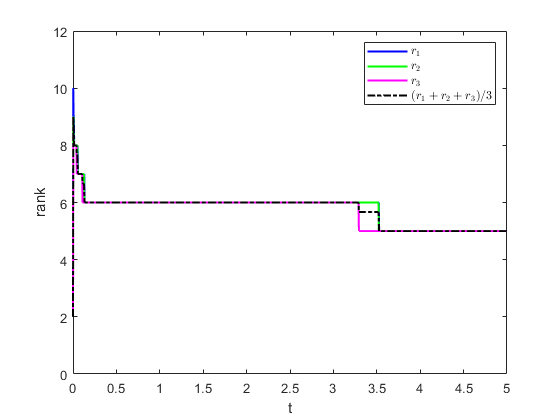}
\end{minipage}
\caption{The multilinear rank and average rank of the solution to \eqref{eq: tests_DFP} with initial condition \eqref{eq: DFP_IC} using IMEX111 \textit{(left)}, IMEX222 \textit{(middle)} and IMEX443 \textit{(right)}.}
\label{fig: DFP_ranktest}
\end{figure}

\begin{figure}[t!]
\centering
\begin{minipage}[b]{0.32\linewidth}
	\centering
	\includegraphics[width=\textwidth]{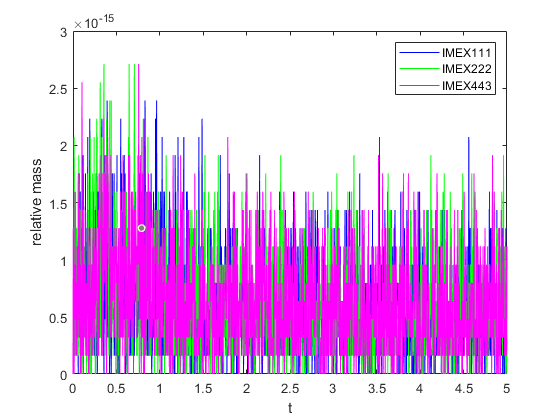}
\end{minipage}
\begin{minipage}[b]{0.32\linewidth}
	\centering
	\includegraphics[width=\textwidth]{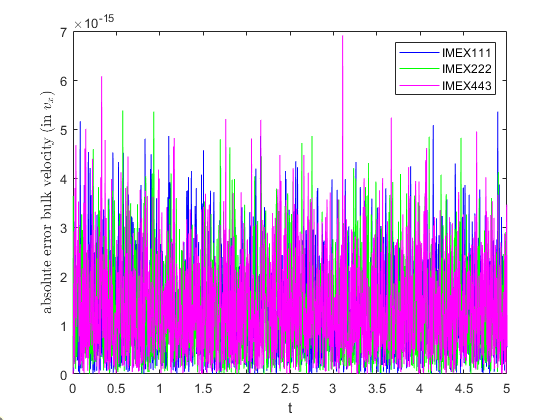}
\end{minipage}
\begin{minipage}[b]{0.32\linewidth}
	\centering
	\includegraphics[width=\textwidth]{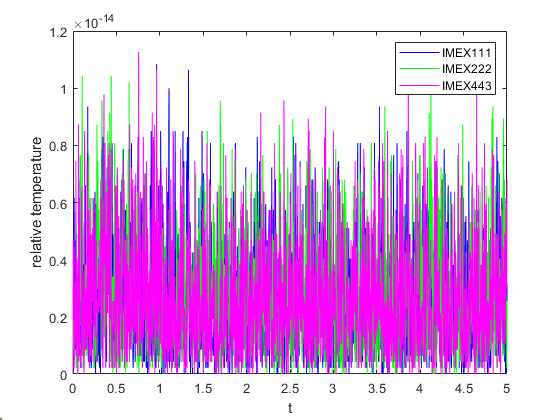}
\end{minipage}
\caption{The conservation of the mass \textit{(left)}, momentum in $v_x$ \textit{(middle)}, and energy \textit{(right)}, for \eqref{eq: tests_DFP} with initial condition \eqref{eq: DFP_IC}.}
\label{fig: DFP_moments}
\end{figure}

\begin{figure}[t!]
\centering
\begin{minipage}[b]{0.48\linewidth}
	\centering
	\includegraphics[width=0.85\textwidth]{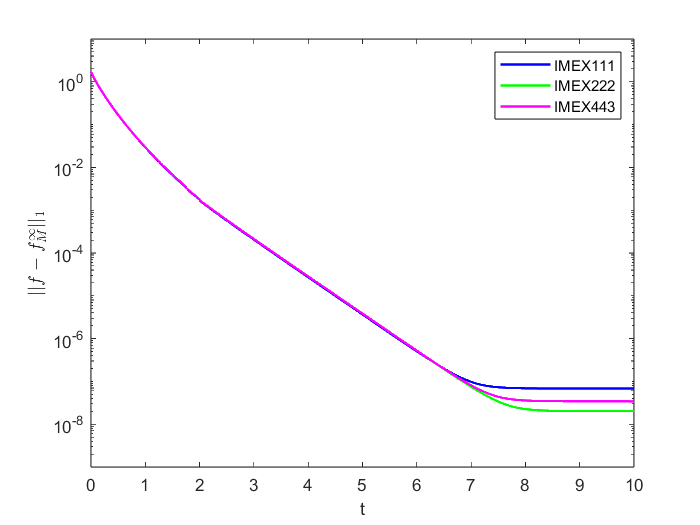}
	\caption{Equilibrium preservation.}
	\label{fig: DFP_L1decay}
\end{minipage}
\begin{minipage}[b]{0.02\linewidth}
\ \\
\end{minipage}
\begin{minipage}[b]{0.48\linewidth}
	\centering
	\includegraphics[width=0.85\textwidth]{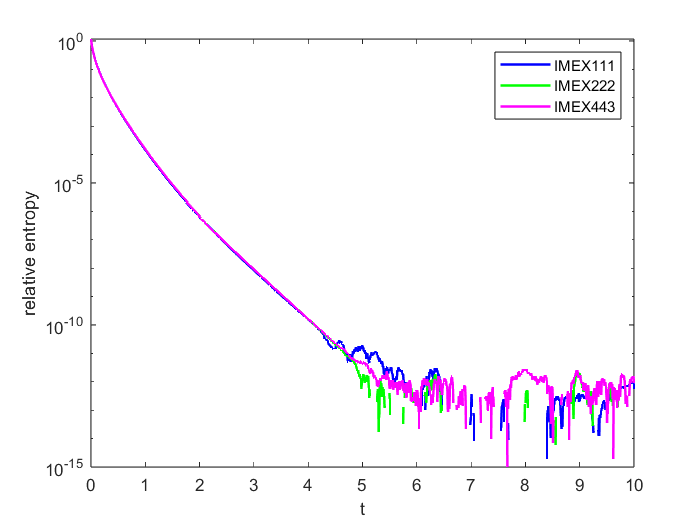}
	\caption{Relative entropy dissipation.}
	\label{fig: DFP_relent}
\end{minipage}
\end{figure}

\subsection[]{Viscous Burgers' equation}
To stress test the 3d-RAIL method, we consider the nonlinear viscous Burgers' equation,
\begin{equation}\label{eq: tests_viscousBurgers}
	u_t + \left(\frac{u^2}{2}\right)_x +  \left(\frac{u^2}{2}\right)_y +  \left(\frac{u^2}{2}\right)_z = d(u_{xx}+u_{yy}+u_{zz}) + c(x,y,z,t),\qquad x,y,z\in(-\pi,\pi).
\end{equation}

We verify the accuracy of the scheme using the manufactured solution $u(x,y,z,t)=e^{-3dt}\sin{(x+y+z)}$ with $d=1/2$, for which the source term
\begin{equation}
	c(x,y,z,t) = \frac{3}{2}e^{-6dt}\sin{(2(x+y+z))}
\end{equation}
can be expressed as a multilinear rank $(2,2,2)$ Tucker tensor using trigonometric identities. As seen in Figure \ref{fig: viscousBurgers}, the expected accuracies are observed when using IMEX111, IMEX222 and IMEX443 with the mass conservative LoMaC truncation (we set $s=5$). Although the solution generally doesn't conserve mass with a source term present, the mass in this case is zero for all time. We used a mesh size $N=100$, tolerance $\varepsilon=10^{-6}$, final time $T_f=0.3$, and $\lambda$ ranging from 0.1 to 1.

Next, we test the 3d-RAIL method on equation \eqref{eq: tests_viscousBurgers} as the solution gradient increases. We let $c=0$ and consider the initial condition $u_0(x,y,z)=\sin{(x+y+z)}$. If $d=0$, then a shock forms at the breaking time $t_b=1/3$. For large enough $d>0$, the diffusion controls the shock and damps the solution to zero as $t\rightarrow\infty$. We vary $d$, for which the solution gradient formed by the intersecting characteristics will be steeper as $d$ decreases. However, as time goes on, the diffusion will eventually dampen the solution. Using a mesh size $N=150$, tolerance $\varepsilon=10^{-5}$, $\lambda=0.9$, and IMEX222 with the mass conservative LoMaC truncation ($s=5$), the average rank $(r_1+r_2+r_3)/3$ and mass are shown in Figure \ref{fig: viscousBurgers} for varying diffusion coefficient $d\in\{1,1/2,1/3,1/4,1/5\}$. The mass is conserved to machine precision, and the multilinear rank increases when the convection is more dominant; using the non-conservative HOSVD truncation slightly reduces the observed multilinear rank, as per Remark \ref{rem: lomac_increasedrank}. Although the multilinear rank is relatively large, which in turn noticeably increases the computational cost due to the flow field being $u/2$, we note that the rank is independent of the mesh. That is, the proposed method becomes advantageous for very fine meshes. We note that the ``steep" gradient forms in all directions and is spherically symmetric. Slices of the solution in the $xy$-plane at $z=z_{75}\approx 0$ are shown in Figure \ref{fig: viscousBurgers_snapshots}, where we observe the solution gradient increasing as $d$ decreases.

\begin{figure}[t!]
\centering
\begin{minipage}[b]{0.32\linewidth}
	\centering
	\includegraphics[width=\textwidth]{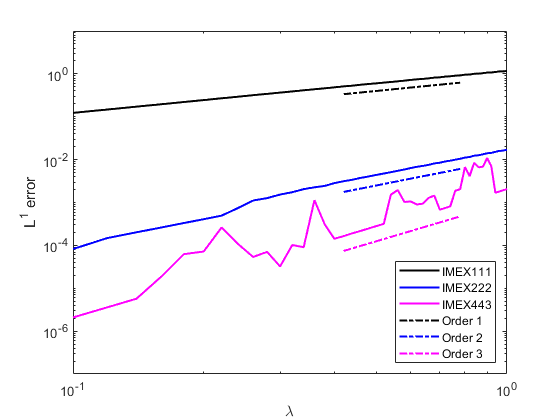}
\end{minipage}
\begin{minipage}[b]{0.32\linewidth}
	\centering
	\includegraphics[width=\textwidth]{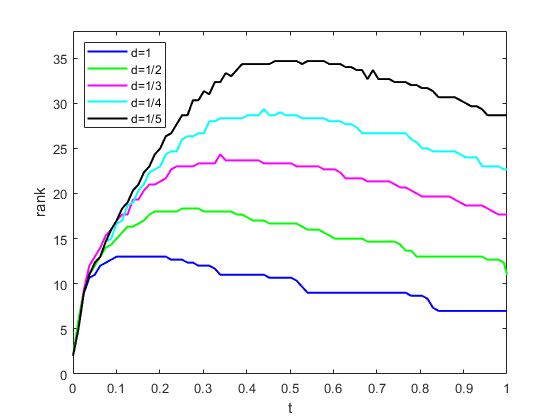}
\end{minipage}
\begin{minipage}[b]{0.32\linewidth}
	\centering
	\includegraphics[width=\textwidth]{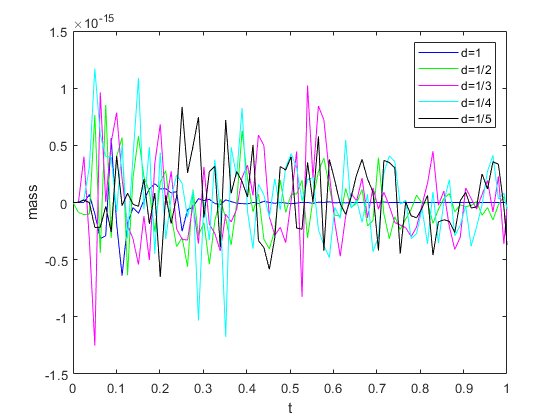}
\end{minipage}
\caption{Viscous Burgers' equation \eqref{eq: tests_viscousBurgers}. \textit{(Left)} Error plot. \textit{(Middle)} Average rank of the solution for $d=1,1/2,1/3,1/4,1/5$. \textit{(Right)} Mass of the solution for $d=1,1/2,1/3,1/4,1/5$.}
\label{fig: viscousBurgers}
\end{figure}

\begin{figure}[t!]
\centering
\begin{minipage}[b]{0.32\linewidth}
	\centering
	\includegraphics[width=\textwidth]{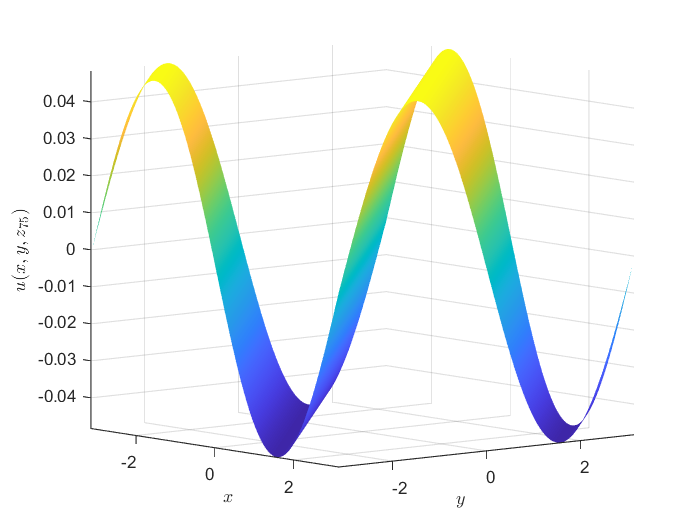}
\end{minipage}
\begin{minipage}[b]{0.32\linewidth}
	\centering
	\includegraphics[width=\textwidth]{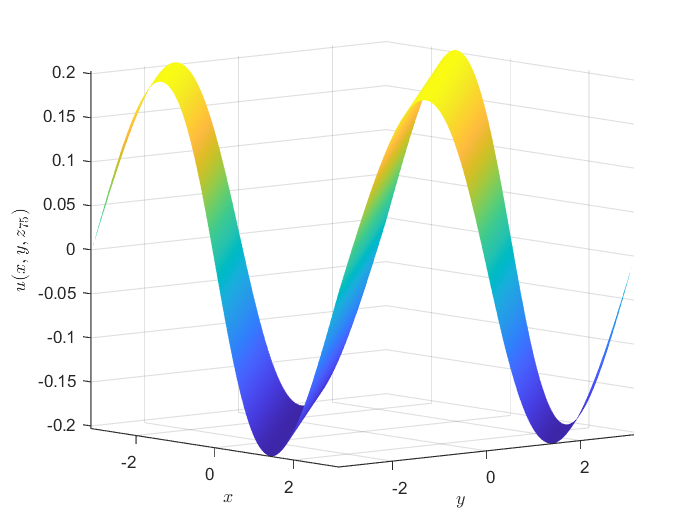}
\end{minipage}
\begin{minipage}[b]{0.32\linewidth}
	\centering
	\includegraphics[width=\textwidth]{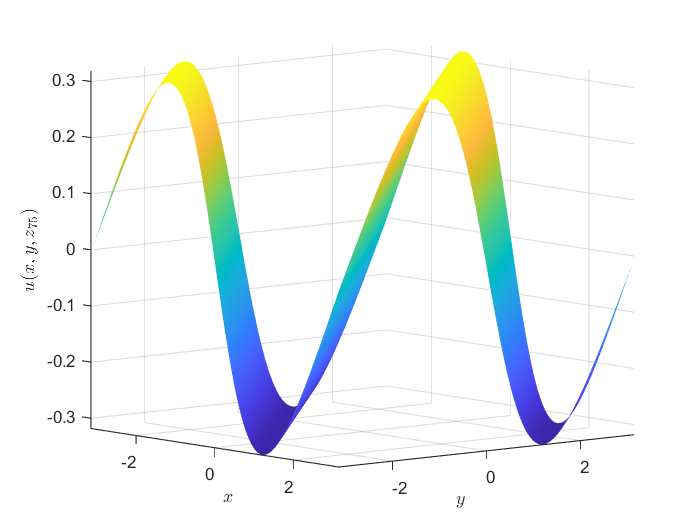}
\end{minipage}
\\
\begin{minipage}[b]{0.32\linewidth}
	\centering
	\includegraphics[width=\textwidth]{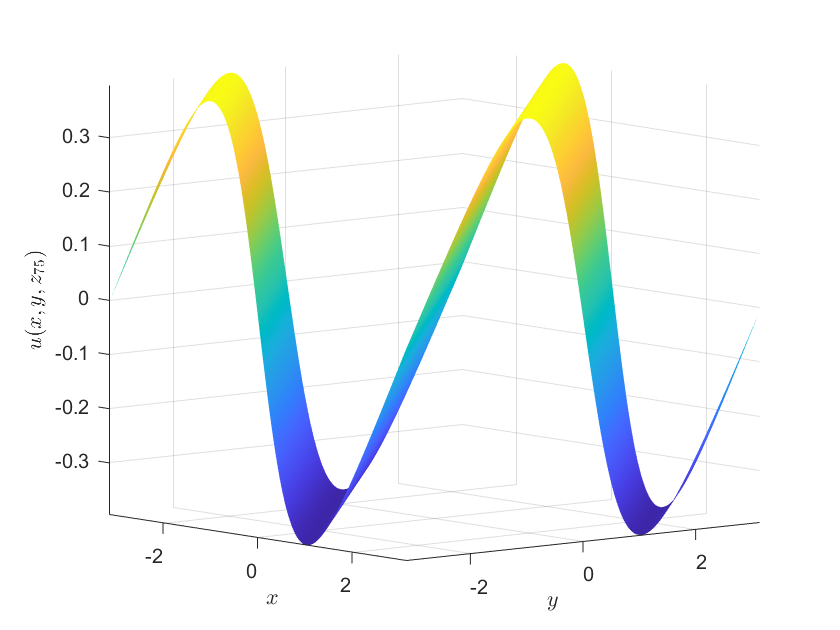}
\end{minipage}
\begin{minipage}[b]{0.32\linewidth}
	\centering
	\includegraphics[width=\textwidth]{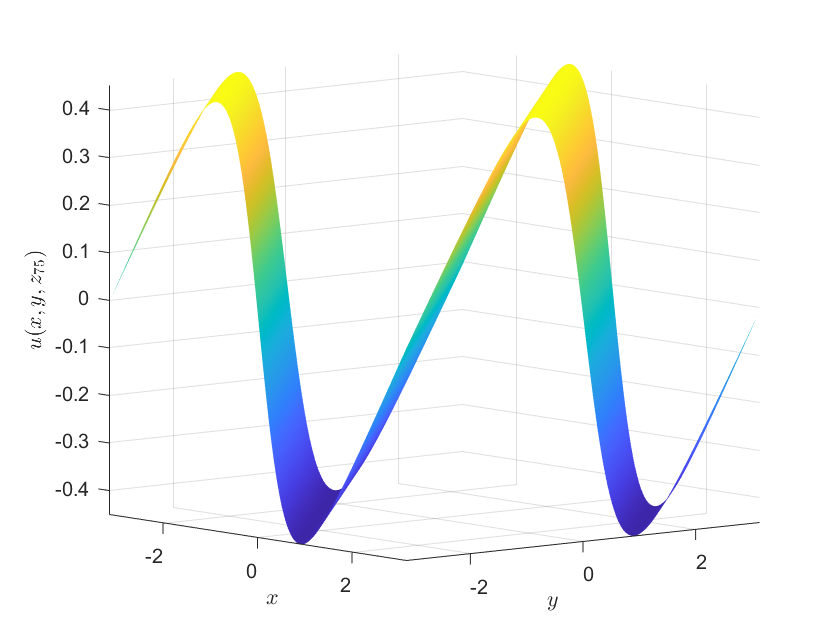}
\end{minipage}
\caption{Solution to viscous Burgers' equation \eqref{eq: tests_viscousBurgers}, slices in the $xy$-plane at $z\approx 0$. \textit{(Top row from left to right)} $d=1$, $d=1/2$, $d=1/3$. \textit{(Bottom row from left to right)} $d=1/4$, $d=1/5$.}
\label{fig: viscousBurgers_snapshots}
\end{figure}

\begin{rem}
Functions with discontinuities or sharp gradients can be low-rank when the sharp gradients align with the coordinate axes. For instance, the rectangular box function $u(x_1,x_2,x_3)=\prod_{i=1}^{3}{H(x_i+1)-H(x_i-1)}$, where $H(x)$ is the Heaviside step function, discretized over a mesh for $\Omega_x=[-2,2]^3$ using the HOSVD with truncation tolerance $\varepsilon=10^{-13}$ is multilinear rank $(1,1,1)$. However, discontinuous or sharp gradient solutions to time-dependent PDEs develop oscillations due to the Gibbs phenomenon. Although total variation diminishing (TVD) or total variation bounded (TVB) methods are commonly used to control spurious oscillations in full-rank methods, extending these methods to the low-rank framework remains a relatively open and ongoing area of research in the scientific community and was beyond the scope of the current paper.
\end{rem}


\section{Conclusion}
In this paper, we proposed a reduced augmentation implicit low-rank (RAIL) integrator for solving three-dimensional convection-diffusion equations. The two-dimensional RAIL algorithm (matrix case) from previous work was extended to three-dimensions using a Tucker decomposition of the third-order tensor solution. The partial differential equations were fully discretized into tensor equations. By spanning the low-order prediction (or updated) bases with the bases from the previous RK stages in a reduced augmentation procedure, the tensor equations were projected onto a richer space that allowed high-order implicit-explicit methods to be used. Several numerical tests demonstrated the proposed RAIL scheme's ability to achieve high-order accuracy using implicit-explicit integrators, capture low-rank structure in solutions, and conserve mass, momentum, and energy. As higher-dimensional models are considered, the recursive structure of higher-dimensional tensor decompositions will rely on efficient three-dimensional solvers such as 3d-RAIL. Ongoing and future work includes extending to hierarchical Tucker decomposition or other tree tensor networks to address dimensions $d\geq 4$, as well as deriving rigorous error bounds for the high-order RAIL scheme.


\section*{Declarations}

\subsection*{Funding} The authors have no funding to report.


\subsection*{Competing Interests} The authors have no competing interests.

\subsection*{Acknowledgements} Joseph Nakao wants to thank Valeria Simoncini for an insightful discussion at SIAM LA24 about her third-order tensor linear equation solver used in the proposed 3d-RAIL integrator.


\appendix

\section{Direct solver from \cite{simoncini2020numerical}}
We summarize the direct solver from \cite{simoncini2020numerical} for efficiently solving third-order tensor linear equations of the form
\begin{equation}\label{eq: 3Dlinearsystem}
	\big(\mathbf{M}_1\otimes\mathbf{A}_1\otimes\mathbf{H}+\mathbf{A}_2\otimes\mathbf{M}\otimes\mathbf{H}+\mathbf{H}_3\otimes\mathbf{M}\otimes\mathbf{A}_3\big)\text{vec}(\X) = \text{vec}{(\B)},
\end{equation}
where all the coefficient matrices are real and have the same $N\times N$ dimensions, $\B$ is a low-rank, and $\mathbf{M}$, $\mathbf{H}$, $\mathbf{M}_1$, $\mathbf{H}_3$, $\A=\mathbf{M}_1\otimes\mathbf{A}_1\otimes\mathbf{H}+\mathbf{A}_2\otimes\mathbf{M}\otimes\mathbf{H}+\mathbf{H}_3\otimes\mathbf{M}\otimes\mathbf{A}_3$ are nonsingular. The proof is constructive, and with the exception of the righthand being a general third-order tensor, is identical to the proof of Theorem 2.1 in \cite{simoncini2020numerical}. Since the algorithm for the direct solver follows naturally from its constructive proof, we provide it for completeness. Simoncini includes MATLAB code for the rank-1 case in \cite{simoncini2020numerical}; for the extended case, small obvious modifications just need to be made.
\begin{cor}[Extension of Theorem 2.1 in \cite{simoncini2020numerical}]\label{cor: 3Dsolver}
Let $\mathbf{A}_3^T\mathbf{H}^{-T} = \mathbf{QRQ}^*$ be the Schur decomposition of $\mathbf{A}_3^T\mathbf{H}^{-T}$, and $[\mathbf{g}_1,\hdots,\mathbf{g}_N]\coloneqq\mathbf{B}_{(1)}^T\mathbf{H}^{-T}\mathbf{Q}$. Using the mode-1 matricization, the solution $\X$ to equation \eqref{eq: 3Dlinearsystem} is given by
\begin{equation}\label{eq: 3Dsolversoln}
\mathbf{X}_{(1)} = \mathbf{Q}[\hat{\mathbf{z}}_1,\hdots,\hat{\mathbf{z}}_N]^T\in\mathbb{R}^{N\times N^2},
\end{equation}
where for $j=1,...,N$, the vector $\hat{\mathbf{z}}_j=\text{vec}(\hat{\mathbf{Z}}_1)$ is the vectorization of the matrix $\hat{\mathbf{Z}}_j$ that solves the Sylvester equation
\begin{equation}\label{eq: SimonciniSolver_sylveqn}
	\mathbf{M}^{-1}\mathbf{A}_1\mathbf{Z} + \mathbf{Z}(R_{j,j}\mathbf{H}_3^T\mathbf{M}_1^{-T}+\mathbf{A}_2^T\mathbf{M}_1^{-T}) = (\mathbf{M}^{-1}\mathbf{G}_j - \mathbf{W}_{j-1}\mathbf{H}_3^T)\mathbf{M}_1^{-T},
\end{equation}
where $R_{j,j}$ denotes the $(j,j)$ element of the upper triangular matrix $\mathbf{R}$, $\mathbf{G}_j$ is the matrix such that $\mathbf{g}_j=\text{vec}(\mathbf{G}_j)$, and $\mathbf{W}_{j-1}$ is the matrix such that $\mathbf{w}_{j-1}=\text{vec}(\mathbf{W}_{j-1})$ with $\mathbf{w}_{j-1}=[\hat{\mathbf{z}}_1,...,\hat{\mathbf{z}}_{j-1}]R_{1:j-1,j}$. We define $\mathbf{W}_0$ to be an empty array for $j=1$.
\end{cor}
\begin{proof}
Equation \eqref{eq: 3Dlinearsystem} can be written as
\begin{equation}
	\X\times_1\mathbf{H}\times_2\mathbf{A}_1\times_3\mathbf{M}_1 + \X\times_1\mathbf{H}\times_2\mathbf{M}\times_3\mathbf{A}_2 + \X\times_1\mathbf{A}_3\times_2\mathbf{M}\times_3\mathbf{H}_3 = \B,
\end{equation}
which after performing the mode-1 matricization can be written as
\begin{equation}\label{eq: 1}
	\mathbf{H}\mathbf{X}_{(1)}(\mathbf{A}_2\otimes\mathbf{M}+\mathbf{M}_1\otimes\mathbf{A}_1)^T + \mathbf{A}_3\mathbf{X}_{(1)}(\mathbf{H}_3\otimes\mathbf{M})^T = \mathbf{B}_{(1)}.
\end{equation}
Multiplying equation \eqref{eq: 1} on the left by $\mathbf{H}^{-1}$ and on the right by $\mathbf{H}_3^{-T}\otimes\mathbf{M}^{-T}$, and letting $\mathbf{Y}=\mathbf{X}_{(1)}^T$,
\begin{equation}\label{eq: 2}
	(\mathbf{H}_3^{-1}\mathbf{A}_2\otimes\mathbf{I} + \mathbf{H}_3^{-1}\mathbf{M}_1\otimes\mathbf{M}^{-1}\mathbf{A}_1)\mathbf{Y} + \mathbf{Y}(\mathbf{H}^{-1}\mathbf{A}_3)^T = (\mathbf{H}_3^{-1}\otimes\mathbf{M}^{-1})\mathbf{B}_{(1)}^T\mathbf{H}^{-T}.
\end{equation}
Using $(\mathbf{H}^{-1}\mathbf{A}_3)^T=\mathbf{QRQ}^*$ and multiplying equation \eqref{eq: 2} on the right by $\mathbf{Q}$, we have
\begin{equation}
	(\mathbf{H}_3^{-1}\mathbf{A}_2\otimes\mathbf{I} + \mathbf{H}_3^{-1}\mathbf{M}_1\otimes\mathbf{M}^{-1}\mathbf{A}_1)\mathbf{YQ} + \mathbf{YQR} = (\mathbf{H}_3^{-1}\otimes\mathbf{M}^{-1})\mathbf{B}_{(1)}^T\mathbf{H}^{-T}\mathbf{Q}.
\end{equation}
Let $\mathbf{YQ}\eqqcolon[\hat{\mathbf{z}}_1,...,\hat{\mathbf{z}}_N]$ and $\mathbf{B}_{(1)}^T\mathbf{H}^{-T}\mathbf{Q}\eqqcolon[\mathbf{g}_1,...,\mathbf{g}_N]$. Thanks to the upper triangular form of $\mathbf{R}$, we have for each $j=1,...,N$,
\begin{equation}\label{eq: 3}
	(\mathbf{H}_3^{-1}\mathbf{A}_2\otimes\mathbf{I} + \mathbf{H}_3^{-1}\mathbf{M}_1\otimes\mathbf{M}^{-1}\mathbf{A}_1)\hat{\mathbf{z}}_j + \hat{\mathbf{z}}_jR_{j,j} = (\mathbf{H}_3^{-1}\otimes\mathbf{M}^{-1})\mathbf{g}_j - \mathbf{w}_{j-1},
\end{equation}
where we set $\mathbf{w}_0=\mathbf{0}$ and $\mathbf{w}_{j-1}=[\hat{\mathbf{z}}_1,...,\hat{\mathbf{z}}_{j-1}]R_{1:j-1,j}$ for $j=2,...,N$. Let $\hat{\mathbf{Z}}_j$, $\mathbf{W}_{j-1}$ and $\mathbf{G}_j$ be the matrices such that $\hat{\mathbf{z}}_j=\text{vec}(\hat{\mathbf{Z}}_j)$, $\hat{\mathbf{w}}_{j-1}=\text{vec}(\hat{\mathbf{W}}_{j-1})$ and $\hat{\mathbf{g}}_j=\text{vec}(\hat{\mathbf{G}}_j)$. Then, after unvectorizing equation \eqref{eq: 3}, we have
\begin{equation}\label{eq: 4}
	\hat{\mathbf{Z}}_j(\mathbf{H}_3^{-1}\mathbf{A}_2)^T + \mathbf{M}^{-1}\mathbf{A}_1\hat{\mathbf{Z}}_j(\mathbf{H}_3^{-1}\mathbf{M}_1)^T + R_{j,j}\hat{\mathbf{Z}}_j = \mathbf{M}^{-1}\mathbf{G}_j\mathbf{H}_3^{-T} - \mathbf{W}_{j-1}.
\end{equation}
Rearranging equation \eqref{eq: 4} and multiplying on the right by $(\mathbf{H}_3^{-1}\mathbf{M}_1)^{-T}$, we have
\begin{equation}
	\mathbf{M}^{-1}\mathbf{A}_1\hat{\mathbf{Z}}_j + \hat{\mathbf{Z}}_j(R_{j,j}\mathbf{H}_3^T\mathbf{M}_1^{-T} + \mathbf{A}_2^T\mathbf{M}_1^{-T}) = \mathbf{M}^{-1}\mathbf{G}_j\mathbf{M}_1^{-T} - \mathbf{W}_{j-1}\mathbf{H}_3^T\mathbf{M}_1^{-T}.
\end{equation}
Solving for $\hat{\mathbf{Z}}_j$ and computing $\hat{\mathbf{z}}_j=\text{vec}(\hat{\mathbf{Z}}_j)$, we get equation \eqref{eq: 3Dsolversoln}.
\end{proof}

\section{Local Macroscopic Conservative (LoMaC) truncation procedure \cite{guo2024local}, extended to third-order Tucker tensor solutions}\label{app: lomac}
We extend the algorithm from \cite{guo2024local} to three-dimensional solutions stored in a low-multilinear rank Tucker tensor decomposition. This procedure is generally applied when evolving a probability distribution function, e.g., in kinetic simulations. Depending on the physical invariants of the model, the zeroth (mass), first (momentum), and/or second (energy) moments of the solution can be conserved. We present the LoMaC algorithm that conserves mass, momentum, and energy. If only mass or mass+momentum conservation is desired, then the following LoMaC procedure can be reduced and simplified as needed. The purpose of this appendix is only to overview the main ideas in order to help the reader follow our publically available code more easily. We refer the reader to the original paper \cite{guo2024local} for specifics and finer details.

Broadly speaking, the idea goes as follows. Scale the distribution function; project the scaled distribution function onto the subspace that conserves the zeroth, first, and second moments; truncate the part in the orthogonal complement since there is zero mass, momentum, and energy. Since the subspace that conserves the moments is of dimension three, the multilinear rank of this part is $(3,3,3)$. The part of the solution that contains zero mass, momentum, and energy, -which is high multilinear rank-, can usually be approximated by a low-multilinear rank tensor. Letting $\mathbf{f}^*\in\mathbb{R}^{N_x\times N_y\times N_z}$ be the (pre-truncated) tensor solution,
\begin{equation}
	\mathbf{f}^* \approx \mathbf{f}_1^{M} + \mathcal{T}_{\varepsilon}(\mathbf{f}_2),
\end{equation}
where $\mathbf{f}_1^{M}$ is the part of the solution that conserves the macroscopic quantites, and $\mathbf{f}_2=(I-P)(\mathbf{f}^*)$ is the part of the solution that contains zero mass, momentum, and energy. The orthogonal projection is performed with respect to the weighted $\ell^2$ inner product $\langle\cdot,\cdot\rangle_{\mathbf{w}}$, where we let $\mathbf{w}$ be a Gaussian distribution that decays fast enough for integrability; other weight functions could be used. In \cite{guo2024local}, the truncated part is scaled and rescaled by $\sqrt{\mathbf{w}}$. However, since $\sqrt{\mathbf{w}}$ is numerically zero near the boundary of the domain, we found this to cause numerical instabilities in our simulations. In the 2d-RAIL paper, the weight function could be perturbed by a small robustness parameter to avoid this issue. In the 3d-RAIL method, we found that removing this scaling and rescaling also resolved the issue; the orthogonal projection that obtains $\mathbf{f}_1^M$ and $\mathbf{f}_2$ was still with respect to the weighted inner product. When truncating $\mathbf{f}_2$, a small amount of artificial mass, momentum and energy is re-introduced into the system as a result of the HOSVD/MLSVD; this mass, momentum and energy is non-physical and from numerical approximation. We simply project out this non-physical part,
\begin{equation}
	\mathbf{f}^*\approx \mathbf{f} \coloneqq \mathbf{f}_1^{M} + (I-P)(\mathcal{T}_{\varepsilon}(\mathbf{f}_2)).
\end{equation}

Deriving the Tucker decompositions of each part follows the same procedure as in \cite{guo2024local}, with $\mathbf{f}_1^M$ being a $(3,3,3)$ multilinear rank Tucker tensor whose core tensor is composed of the macroscopic quantities. Unlike the step-and-truncate method in \cite{guo2024local}, we need the factor matrices of the Tucker decomposition for $\mathbf{f}$ to be orthonormal in the \textit{unweighted} $\ell^2$ inner product; the weight inner product was just for the LoMaC projection. As such, we orthonormalize the factor matrices using a reduced QR factorization, followed by performing the mode-$n$ product between the core tensor and the upper triangular $R$ matrices from the QR factorizations. We refer the reader to \cite{guo2024local} and our publically provided code for more details.

\clearpage
\section{Butcher tables for implicit-explicit Runge-Kutta methods \cite{Ascher1997}}

\begin{table}[h!]
    \centering
    \begin{minipage}[b]{0.49\linewidth}
    \centering
    \caption{IMEX(1,1,1) -- Implicit Table}
    \begin{tabular}{c|ll}
        0&0&0\\
        1&0&1\\
        \hline
        &0&1
        \end{tabular}
    \end{minipage}
    \begin{minipage}[b]{0.49\linewidth}
    \centering
    \caption{IMEX(1,1,1) -- Explicit Table}
    \begin{tabular}{c|ll}
        0&0&0\\
        1&1&0\\
        \hline
        &1&0
        \end{tabular}
    \end{minipage}
\end{table}

\begin{table}[h!]
    \centering
    \begin{minipage}[b]{0.49\linewidth}
    \centering
    \caption{IMEX(2,2,2) -- Implicit Table}
    \begin{tabular}{c|lll}
        0&0&0&0\\
        $\nu$&0&$\nu$&0\\
        1&0&$1-\nu$&$\nu$\\
        \hline
        &0&$1-\nu$&$\nu$
        \end{tabular}
    \end{minipage}
    \begin{minipage}[b]{0.49\linewidth}
    \centering
    \caption{IMEX(2,2,2) -- Explicit Table}
    \begin{tabular}{c|lll}
        0&0&0&0\\
        $\nu$&$\nu$&0&0\\
        1&$\delta$&$1-\delta$&0\\
        \hline
        &$\delta$&$1-\delta$&0
        \end{tabular}
    \end{minipage}
    Let $\nu=1-\sqrt{2}/2$ and $\delta = 1 - 1/(2\nu)$.
\end{table}

\begin{table}[h!]
    \centering
    \begin{minipage}[b]{0.49\linewidth}
    \centering
    \caption{IMEX(4,4,3) -- Implicit Table}
    \begin{tabular}{c|lllll}
        0&0&0&0&0&0\\
        1/2&0&1/2&0&0&0\\
        2/3&0&1/6&1/2&0&0\\
        1/2&0&-1/2&1/2&1/2&0\\
        1&0&3/2&-3/2&1/2&1/2\\
        \hline
        &0&3/2&-3/2&1/2&1/2
        \end{tabular}
    \end{minipage}
    \begin{minipage}[b]{0.49\linewidth}
    \centering
    \caption{IMEX(4,4,3) -- Explicit Table}
    \begin{tabular}{c|lllll}
        0&0&0&0&0&0\\
        1/2&1/2&0&0&0&0\\
        2/3&11/18&1/18&0&0&0\\
        1/2&5/6&-5/6&1/2&0&0\\
        1&1/4&7/4&3/4&-7/4&0\\
        \hline
        &1/4&7/4&3/4&-7/4&0
        \end{tabular}
    \end{minipage}
\end{table}

\medskip

\printbibliography

@article{einkemmer2025review,
  title={A review of low-rank methods for time-dependent kinetic simulations},
  author={Einkemmer, Lukas and Kormann, Katharina and Kusch, Jonas and McClarren, Ryan G and Qiu, Jing-Mei},
  journal={Journal of Computational Physics},
  pages={114191},
  year={2025},
  publisher={Elsevier}
}

@Article{taitano_jcp_2017_ep_rfp,
    Author = {Taitano, W T and Chac{\'o}n, L and Simakov, A N},
    Title = {An equilibrium-preserving discretization for the nonlinear {R}osenbluth-{F}okker-{P}lanck operator in arbitrary multi-dimensional geometry},
    Journal = {J. Comp. Phys.},
    Volume = {339},
    number = {},
    Pages = {453-460},
    Year = {2017}
}

@Article{Cassini2022,
  author    = {Cassini, F. and Einkemmer, L.},
  journal   = {einkemmer huComput. Phys. Commun.},
  title     = {{Efficient 6D Vlasov simulation using the dynamical low-rank framework Ensign}},
  year      = {2022},
  pages     = {108489},
  volume    = {280},
  publisher = {Elsevier},
}

@Article{Einkemmer2018,
  author  = {Einkemmer, L. and Lubich, C.},
  journal = {SIAM J. Sci. Comput.},
  title   = {A Low-Rank Projector-Splitting Integrator for the {V}lasov--{P}oisson Equation},
  year    = {2018},
  pages   = {B1330-B1360},
  volume  = {40},
}

@Article{Meyer1990,
  author    = {H.-D. Meyer and U. Manthe and L. S. Cederbaum},
  journal   = {Chem. Phys. Letters},
  title     = {The multi-configurational time-dependent {H}artree approach},
  year      = {1990},
  pages     = {73--78},
  volume    = {165},
  publisher = {Elsevier},
}

@book{Lubich2008,
	author = {Lubich, C.},
	publisher = {European Mathematical Society},
	title = {{From quantum to classical molecular dynamics: reduced models and numerical analysis}},
	year = {2008},
	address = {Z\"{u}rich}
}

@phdthesis{kurschner2016efficient,
	title={Efficient low-rank solution of large-scale matrix equations},
	author={K{\"u}rschner, Patrick},
	year={2016},
	school={Shaker Verlag Aachen}
}

@article{kressner2015truncated,
	title={Truncated low-rank methods for solving general linear matrix equations},
	author={Kressner, Daniel and Sirkovi{\'c}, Petar},
	journal={Numerical Linear Algebra with Applications},
	volume={22},
	number={3},
	pages={564--583},
	year={2015},
	publisher={Wiley Online Library}
}

@article{hackbusch2009new,
	title={A new scheme for the tensor representation},
	author={Hackbusch, Wolfgang and K{\"u}hn, Stefan},
	journal={Journal of Fourier analysis and applications},
	volume={15},
	number={5},
	pages={706--722},
	year={2009},
	publisher={Springer}
}

@article{ceruti2021time,
	title={Time integration of tree tensor networks},
	author={Ceruti, Gianluca and Lubich, Christian and Walach, Hanna},
	journal={SIAM Journal on Numerical Analysis},
	volume={59},
	number={1},
	pages={289--313},
	year={2021},
	publisher={SIAM}
}

@article{ceruti2023rank,
	title={Rank-adaptive time integration of tree tensor networks},
	author={Ceruti, Gianluca and Lubich, Christian and Sulz, Dominik},
	journal={SIAM Journal on Numerical Analysis},
	volume={61},
	number={1},
	pages={194--222},
	year={2023},
	publisher={SIAM}
}

@article{PengDLR2020spherical,
    author = {Zhuogang Peng and Ryan G. McClarren and Martin Frank},
    title = {A low-rank method for two-dimensional time-dependent radiation transport calculations},
    journal={Journal of Computational Physics},
    volume = {421},
    pages = {109735},
    year = {2020}
}

@Article{Coughlin2022,
  author    = {J. Coughlin and J. Hu},
  journal   = {J. Comput. Phys.},
  title     = {{Efficient dynamical low-rank approximation for the Vlasov-Amp{\`{e}}re-Fokker-Planck system}},
  year      = {2022},
  pages     = {111590},
  volume    = {470},
  doi       = {10.1016/j.jcp.2022.111590},
  publisher = {Elsevier {BV}},
}

@article{Kusch2022power,
  title={A low-rank power iteration scheme for neutron transport criticality problems},
  author={Kusch, Jonas and Whewell, Benjamin and McClarren, Ryan and Frank, Martin},
  journal={Journal of Computational Physics},
  volume={470},
  pages={111587},
  year={2022},
  publisher={Elsevier}
}

@Article{Lubich2015a,
  author    = {C. Lubich},
  journal   = {Appl. Math. Res. Express. AMRX},
  title     = {Time integration in the multiconfiguration time-dependent {H}artree method of molecular quantum dynamics},
  year      = {2015},
  pages     = {311--328},
  volume    = {2015},
  publisher = {Oxford University Press},
}

@Book{Meyer2009,
  author    = {H.-D. Meyer and F. Gatti and G. A. Worth},
  publisher = {John Wiley \& Sons},
  title     = {Multidimensional quantum dynamics},
  year      = {2009},
}

@Article{Jahnke2008,
  author    = {T. Jahnke and W. Huisinga},
  journal   = {Bull. Math. Biol.},
  title     = {A Dynamical Low-Rank Approach to the Chemical Master Equation},
  year      = {2008},
  number    = {8},
  pages     = {2283--2302},
  volume    = {70},
  doi       = {10.1007/s11538-008-9346-x},
  publisher = {Springer Science and Business Media {LLC}},
}

@Article{Kusch2021radiation,
  author    = {Kusch, J. and Stammer, P.},
  journal   = {ESAIM: Mathematical Modelling and Numerical Analysis},
  title     = {{A robust collision source method for rank adaptive dynamical low-rank approximation in radiation therapy}},
  year      = {2023},
  number    = {2},
  pages     = {865--891},
  volume    = {57},
  publisher = {EDP Sciences},
}

@Article{Prugger2023,
  author    = {M. Prugger and L. Einkemmer and C.F. Lopez},
  journal   = {J. Comput. Phys.},
  title     = {A dynamical low-rank approach to solve the chemical master equation for biological reaction networks},
  year      = {2023},
  pages     = {112250},
  doi       = {10.1016/j.jcp.2023.112250},
  publisher = {Elsevier {BV}},
  vol       = {489},
}

@article{einkemmer2024chemical,
  title={A low-rank complexity reduction algorithm for the high-dimensional kinetic chemical master equation},
  author={Einkemmer, Lukas and Mangott, Julian and Prugger, Martina},
  journal={Journal of Computational Physics},
  volume={503},
  pages={112827},
  year={2024},
  publisher={Elsevier}
}

@article{einkemmer2024reaction,
  title={A hierarchical dynamical low-rank algorithm for the stochastic description of large reaction networks},
  author={Einkemmer, Lukas and Mangott, Julian and Prugger, Martina},
  journal={arXiv preprint arXiv:2407.11792},
  year={2024}
}

@article{naderi2025cross,
  title={A cross algorithm for implicit time integration of random partial differential equations on low-rank matrix manifolds},
  author={Naderi, Mohammad Hossein and Akhavan, Sara and Babaee, Hessam},
  journal={Proceedings of the Royal Society A},
  volume={481},
  number={2309},
  pages={20240658},
  year={2025},
  publisher={The Royal Society}
}

@article{sutti2024implicit,
	title={Implicit low-rank Riemannian schemes for the time integration of stiff partial differential equations},
	author={Sutti, Marco and Vandereycken, Bart},
	journal={Journal of Scientific Computing},
	volume={101},
	number={1},
	pages={3},
	year={2024},
	publisher={Springer}
}

@article{appelo2025robust,
	title={Robust implicit adaptive low rank time-stepping methods for matrix differential equations},
	author={Appel{\"o}, Daniel and Cheng, Yingda},
	journal={Journal of Scientific Computing},
	volume={102},
	number={3},
	pages={81},
	year={2025},
	publisher={Springer}
}

@article{el2024krylov,
	title={Krylov-based adaptive-rank implicit time integrators for stiff problems with application to nonlinear Fokker-Planck kinetic models},
	author={El Kahza, Hamad and Taitano, William and Qiu, Jing-Mei and Chac{\'o}n, Luis},
	journal={Journal of Computational Physics},
	volume={518},
	pages={113332},
	year={2024},
	publisher={Elsevier}
}

@article{kahza2024sylvester,
	title={Sylvester-Preconditioned Adaptive-Rank Implicit Time Integrators for Advection-Diffusion Equations with Inhomogeneous Coefficients},
	author={El Kahza, Hamad and Qiu, Jing-Mei and Chacon, Luis and Taitano, William},
	journal={arXiv preprint arXiv:2410.19662},
	year={2024}
}

@article{rodgers2023implicit,
	title={Implicit integration of nonlinear evolution equations on tensor manifolds},
	author={Rodgers, Abram and Venturi, Daniele},
	journal={Journal of Scientific Computing},
	volume={97},
	number={2},
	pages={33},
	year={2023},
	publisher={Springer}
}

@article{li2024high,
  title={High-Order Implicit Low-Rank Method with Spectral Deferred Correction for Matrix Differential Equations},
  author={Li, Shun and Jiang, Yan and Cheng, Yingda},
  journal={arXiv preprint arXiv:2412.09400},
  year={2024}
}

@article{meng2024preconditioning,
  title={Preconditioning low rank generalized minimal residual method (gmres) for implicit discretizations of matrix differential equations},
  author={Meng, Shixu and Appelo, Daniel and Cheng, Yingda},
  journal={arXiv preprint arXiv:2410.07465},
  year={2024}
}

@article{ding2021dynamical,
  title={Dynamical low-rank integrator for the linear Boltzmann equation: error analysis in the diffusion limit},
  author={Ding, Zhiyan and Einkemmer, Lukas and Li, Qin},
  journal={SIAM Journal on Numerical Analysis},
  volume={59},
  number={4},
  pages={2254--2285},
  year={2021},
  publisher={SIAM}
}

@article{einkemmer2021asymptotic,
  title={An asymptotic-preserving dynamical low-rank method for the multi-scale multi-dimensional linear transport equation},
  author={Einkemmer, Lukas and Hu, Jingwei and Wang, Yubo},
  journal={Journal of Computational Physics},
  volume={439},
  pages={110353},
  year={2021},
  publisher={Elsevier}
}

@article{sands2024high,
  title={High-order adaptive rank integrators for multi-scale linear kinetic transport equations in the hierarchical tucker format},
  author={Sands, William A and Guo, Wei and Qiu, Jing-Mei and Xiong, Tao},
  journal={arXiv preprint arXiv:2406.19479},
  year={2024}
}

@article{frank2025asymptotic,
  title={Asymptotic-preserving and energy stable dynamical low-rank approximation for thermal radiative transfer equations},
  author={Frank, Martin and Kusch, Jonas and Patwardhan, Chinmay},
  journal={Multiscale Modeling \& Simulation},
  volume={23},
  number={1},
  pages={278--312},
  year={2025},
  publisher={SIAM}
}

@article{Kieri2016,
  title={{Discretized dynamical low-rank approximation in the presence of small singular values}},
  author={E. Kieri and C. Lubich and H. Walach},
  journal={SIAM J. Numer. Anal.},
  volume={54},
  number={2},
  pages={1020--1038},
  year={2016}
}

@article{Nonnenmacher2008,
  title={Dynamical low-rank approximation: applications and numerical experiments},
  author={Nonnenmacher, A. and Lubich, C.},
  journal={Math. Comput. Simul.},
  volume={79},
  number={4},
  pages={1346--1357},
  year={2008},
  publisher={Elsevier}
}

@article{Carrel2023,
  title={Projected exponential methods for stiff dynamical low-rank approximation problems},
  author={Carrel, Benjamin and Vandereycken, Bart},
  journal={arXiv preprint arXiv:2312.00172},
  year={2023}
}

@article{charous2023dynamically,
  title={Dynamically orthogonal Runge--Kutta schemes with perturbative retractions for the dynamical low-rank approximation},
  author={Charous, Aaron and Lermusiaux, Pierre FJ},
  journal={SIAM Journal on Scientific Computing},
  volume={45},
  number={2},
  pages={A872--A897},
  year={2023},
  publisher={SIAM}
}

@article{kieri2019projection,
  title={Projection methods for dynamical low-rank approximation of high-dimensional problems},
  author={Kieri, Emil and Vandereycken, Bart},
  journal={Computational Methods in Applied Mathematics},
  volume={19},
  number={1},
  pages={73--92},
  year={2019},
  publisher={De Gruyter}
}

@article{ceruti2024robust,
	title={A robust second-order low-rank BUG integrator based on the midpoint rule},
	author={Ceruti, Gianluca and Einkemmer, Lukas and Kusch, Jonas and Lubich, Christian},
	journal={BIT Numerical Mathematics},
	volume={64},
	number={3},
	pages={30},
	year={2024},
	publisher={Springer}
}

@article{seguin2024low,
	title={From low-rank retractions to dynamical low-rank approximation and back},
	author={S{\'e}guin, Axel and Ceruti, Gianluca and Kressner, Daniel},
	journal={BIT Numerical Mathematics},
	volume={64},
	number={3},
	pages={25},
	year={2024},
	publisher={Springer}
}

@article{kusch2025second,
  title={Second-order robust parallel integrators for dynamical low-rank approximation},
  author={Kusch, Jonas},
  journal={BIT Numerical Mathematics},
  volume={65},
  number={3},
  pages={31},
  year={2025},
  publisher={Springer}
}

@article{nobile2025robust,
	title={Robust high-order low-rank BUG integrators based on explicit Runge-Kutta methods},
	author={Nobile, Fabio and Riffaud, S{\'e}bastien},
	journal={arXiv preprint arXiv:2502.07040},
	year={2025}
}

@article{koch2007dynamical,
	title={Dynamical low-rank approximation},
	author={Koch, Othmar and Lubich, Christian},
	journal={SIAM Journal on Matrix Analysis and Applications},
	volume={29},
	number={2},
	pages={434--454},
	year={2007},
	publisher={SIAM}
}

@article{koch2010dynamical,
	title={Dynamical tensor approximation},
	author={Koch, Othmar and Lubich, Christian},
	journal={SIAM Journal on Matrix Analysis and Applications},
	volume={31},
	number={5},
	pages={2360--2375},
	year={2010},
	publisher={SIAM}
}

@article{yang2025second,
  title={A second-order dynamical low-rank mass-lumped finite element method for the Allen-Cahn equation},
  author={Yang, Jun and Yi, Nianyu and Yin, Peimeng},
  journal={arXiv preprint arXiv:2501.06145},
  year={2025}
}

@article{donello2023oblique,
	title={Oblique projection for scalable rank-adaptive reduced-order modelling of nonlinear stochastic partial differential equations with time-dependent bases},
	author={Donello, M and Palkar, G and Naderi, MH and Del Rey Fern{\'a}ndez, DC and Babaee, H},
	journal={Proceedings of the Royal Society A},
	volume={479},
	number={2278},
	pages={20230320},
	year={2023},
	publisher={The Royal Society}
}

@article{dektor2025collocation,
	title={Collocation methods for nonlinear differential equations on low-rank manifolds},
	author={Dektor, Alec},
	journal={Linear Algebra and its Applications},
	volume={705},
	pages={143--184},
	year={2025},
	publisher={Elsevier}
}

@article{lam2024randomized,
	title={Randomized low-rank runge-kutta methods},
	author={Lam, Hei Yin and Ceruti, Gianluca and Kressner, Daniel},
	journal={arXiv preprint arXiv:2409.06384},
	year={2024}
}

@misc{carrel2024randDLR,
	title={Randomised methods for dynamical low-rank approximation}, 
	author={Benjamin Carrel},
	year={2024},
	eprint={2410.17091},
	archivePrefix={arXiv},
	primaryClass={math.NA},
	url={https://arxiv.org/abs/2410.17091}, 
}

@article{ghahremani2024deim,
	title={A DEIM Tucker tensor cross algorithm and its application to dynamical low-rank approximation},
	author={Ghahremani, Behzad and Babaee, Hessam},
	journal={Computer Methods in Applied Mechanics and Engineering},
	volume={423},
	pages={116879},
	year={2024},
	publisher={Elsevier}
}

@article{dektor2024interpolatory,
	title={Interpolatory dynamical low-rank approximation for the 3+ 3d Boltzmann-BGK equation},
	author={Dektor, Alec and Einkemmer, Lukas},
	journal={arXiv preprint arXiv:2411.15990},
	year={2024}
}

@article{guo2024conservative,
	title={A conservative low rank tensor method for the Vlasov dynamics},
	author={Guo, Wei and Qiu, Jing-Mei},
	journal={SIAM Journal on Scientific Computing},
	volume={46},
	number={1},
	pages={A232--A263},
	year={2024},
	publisher={SIAM}
}

@article{guo2024local,
	title={A Local Macroscopic Conservative (LoMaC) low rank tensor method for the Vlasov dynamics},
	author={Guo, Wei and Qiu, Jing-Mei},
	journal={Journal of Scientific Computing},
	volume={101},
	number={3},
	pages={61},
	year={2024},
	publisher={Springer}
}

@article{rodgers2020step,
	title={Step-truncation integrators for evolution equations on low-rank tensor manifolds},
	author={Rodgers, Abram and Venturi, Daniele},
	journal={CoRR},
	year={2020}
}

@article{rodgers2024tensor,
	title={Tensor approximation of functional differential equations},
	author={Rodgers, Abram and Venturi, Daniele},
	journal={Physical Review E},
	volume={110},
	number={1},
	pages={015310},
	year={2024},
	publisher={APS}
}

@article{rodgers2022adaptive,
	title={Adaptive integration of nonlinear evolution equations on tensor manifolds},
	author={Rodgers, Abram and Dektor, Alec and Venturi, Daniele},
	journal={Journal of Scientific Computing},
	volume={92},
	number={2},
	pages={39},
	year={2022},
	publisher={Springer}
}

@article{dolgov2014low,
  title={Low-rank approximation in the numerical modeling of the Farley--Buneman instability in ionospheric plasma},
  author={Dolgov, Sergey V and Smirnov, Alexander P and Tyrtyshnikov, EE},
  journal={Journal of Computational Physics},
  volume={263},
  pages={268--282},
  year={2014},
  publisher={Elsevier}
}

@article{ehrlacher2017dynamical,
  title={A dynamical adaptive tensor method for the Vlasov--Poisson system},
  author={Ehrlacher, Virginie and Lombardi, Damiano},
  journal={Journal of Computational Physics},
  volume={339},
  pages={285--306},
  year={2017},
  publisher={Elsevier}
}

@article{de2000multilinear,
  title={A multilinear singular value decomposition},
  author={De Lathauwer, Lieven and De Moor, Bart and Vandewalle, Joos},
  journal={SIAM journal on Matrix Analysis and Applications},
  volume={21},
  number={4},
  pages={1253--1278},
  year={2000},
  publisher={SIAM}
}

@inproceedings{vervliet2016tensorlab_confpaper,
  title={Tensorlab 3.0—numerical optimization strategies for large-scale constrained and coupled matrix/tensor factorization},
  author={Vervliet, Nico and Debals, Otto and De Lathauwer, Lieven},
  booktitle={2016 50th Asilomar Conference on Signals, Systems and Computers},
  pages={1733--1738},
  year={2016},
  organization={IEEE}
}

@article{vervliet2016tensorlab,
  title={Tensorlab 3.0},
  author={Vervliet, Nico and Debals, Otto and Sorber, Laurent and Van Barel, Marc and De Lathauwer, Lieven},
  journal={Available online: http://www. tensorlab. net},
  year={2016}
}

@article{grasedyck2013literature,
	title={A literature survey of low-rank tensor approximation techniques},
	author={Grasedyck, Lars and Kressner, Daniel and Tobler, Christine},
	journal={GAMM-Mitteilungen},
	volume={36},
	number={1},
	pages={53--78},
	year={2013},
	publisher={Wiley Online Library}
}

@article{kolda2009tensor,
  title={Tensor decompositions and applications},
  author={Kolda, Tamara G and Bader, Brett W},
  journal={SIAM review},
  volume={51},
  number={3},
  pages={455--500},
  year={2009},
  publisher={SIAM}
}

@inproceedings{kormann2017low,
%  title={Low-rank tensor discretization for high-dimensional problems},
%  author={Kormann, Katharina},
%  booktitle={Vorlesung (SS 2017)},
%  year={2017}
%}

@article{kressner2017recompression,
  title={Recompression of Hadamard products of tensors in Tucker format},
  author={Kressner, Daniel and Perisa, Lana},
  journal={SIAM Journal on Scientific Computing},
  volume={39},
  number={5},
  pages={A1879--A1902},
  year={2017},
  publisher={SIAM}
}

@phdthesis{ragnarsson2012structured,
  title={Structured tensor computations: Blocking, symmetries and Kronecker factorizations},
  author={Ragnarsson, Stefan},
  year={2012},
  school={Cornell University}
}

@article{lee2014fundamental,
  title={Fundamental tensor operations for large-scale data analysis in tensor train formats},
  author={Lee, Namgil and Cichocki, Andrzej},
  journal={arXiv preprint arXiv:1405.7786},
  year={2014}
}

@article{tucker1963implications,
  title={Implications of factor analysis of three-way matrices for measurement of change},
  author={Tucker, Ledyard R},
  journal={Problems in measuring change},
  volume={15},
  number={122-137},
  pages={3},
  year={1963},
  publisher={University of Wisconsin Press Madison}
}

@article{tucker1966some,
  title={Some mathematical notes on three-mode factor analysis},
  author={Tucker, Ledyard R},
  journal={Psychometrika},
  volume={31},
  number={3},
  pages={279--311},
  year={1966},
  publisher={Springer}
}

@article{kapteyn1986approach,
  title={An approach to n-mode components analysis},
  author={Kapteyn, Arie and Neudecker, Heinz and Wansbeek, Tom},
  journal={Psychometrika},
  volume={51},
  pages={269--275},
  year={1986},
  publisher={Springer}
}

@article{oseledets2011tensor,
  title={Tensor-train decomposition},
  author={Oseledets, Ivan V},
  journal={SIAM Journal on Scientific Computing},
  volume={33},
  number={5},
  pages={2295--2317},
  year={2011},
  publisher={SIAM}
}

@article{simoncini2020numerical,
  title={Numerical solution of a class of third order tensor linear equations},
  author={Simoncini, V},
  journal={Bollettino dell'Unione Matematica Italiana},
  volume={13},
  number={3},
  pages={429--439},
  year={2020},
  publisher={Springer}
}

@article{golub1979hessenberg,
  title={A Hessenberg-Schur method for the problem AX+ XB= C},
  author={Golub, Gene and Nash, Stephen and Van Loan, Charles},
  journal={IEEE Transactions on Automatic Control},
  volume={24},
  number={6},
  pages={909--913},
  year={1979},
  publisher={IEEE}
}

@article{bartels1972,
  title={Solution of the matrix equation AX+ XB= C [F4]},
  author={Bartels, Richard H. and Stewart, George W},
  journal={Communications of the ACM},
  volume={15},
  number={9},
  pages={820--826},
  year={1972},
  publisher={Association for Computing Machinery (ACM)}
}

@article{lubich2014projector,
	title={A projector-splitting integrator for dynamical low-rank approximation},
	author={Lubich, Christian and Oseledets, Ivan V},
	journal={BIT Numerical Mathematics},
	volume={54},
	number={1},
	pages={171--188},
	year={2014},
	publisher={Springer}
}

@article{ceruti2024parallel,
	title={A parallel rank-adaptive integrator for dynamical low-rank approximation},
	author={Ceruti, Gianluca and Kusch, Jonas and Lubich, Christian},
	journal={SIAM Journal on Scientific Computing},
	volume={46},
	number={3},
	pages={B205--B228},
	year={2024},
	publisher={SIAM}
}

@article{ceruti2022unconventional,
	title={An unconventional robust integrator for dynamical low-rank approximation},
	author={Ceruti, Gianluca and Lubich, Christian},
	journal={BIT Numerical Mathematics},
	volume={62},
	number={1},
	pages={23--44},
	year={2022},
	publisher={Springer}
}

@article{ceruti2022rank,
  title={A rank-adaptive robust integrator for dynamical low-rank approximation},
  author={Ceruti, Gianluca and Kusch, Jonas and Lubich, Christian},
  journal={BIT Numerical Mathematics},
  volume={62},
  number={4},
  pages={1149--1174},
  year={2022},
  publisher={Springer}
}

@article{nakao2025reduced,
  title={Reduced Augmentation Implicit Low-rank (RAIL) integrators for advection-diffusion and Fokker--Planck models},
  author={Nakao, Joseph and Qiu, Jing-Mei and Einkemmer, Lukas},
  journal={SIAM Journal on Scientific Computing},
  volume={47},
  number={2},
  pages={A1145--A1169},
  year={2025},
  publisher={SIAM}
}

@book{nakao2023thesis,
  title={Speeding up high-order algorithms in computational fluid and kinetic dynamics: Based on characteristics tracing and low-rank structures},
  author={Nakao, Joseph},
  year={2023},
  publisher={University of Delaware}
}

@Article{Ascher1997,
  author  = {Ascher, U. M. and Ruuth, S. J. and Spiteri, R. J.},
  journal = {Applied Numerical Mathematics},
  title   = {{Implicit-explicit Runge-Kutta methods for time-dependent partial differential equations}},
  year    = {1997},
  number  = {2-3},
  pages   = {151--167},
  volume  = {25},
}

@book{Trefethen2000,
  title={Spectral methods in MATLAB},
  author={Trefethen, Lloyd N},
  year={2000},
  publisher={SIAM}
}

@article{wang2015stability,
  title={Stability and error estimates of local discontinuous Galerkin methods with implicit-explicit time-marching for advection-diffusion problems},
  author={Wang, Haijin and Shu, Chi-Wang and Zhang, Qiang},
  journal={SIAM Journal on Numerical Analysis},
  volume={53},
  number={1},
  pages={206--227},
  year={2015},
  publisher={SIAM}
}

\end{document}